\documentclass[11pt]{amsart}
\usepackage{amsmath,amsthm,amsfonts,latexsym,amssymb}
\usepackage{graphicx}
\usepackage[arrow, matrix, curve]{xy}

\calclayout\evensidemargin .1in

\newcommand{\R}{\mathbb{R}}

\newtheorem{theorem}{Theorem}[section]
\newtheorem{lemma}[theorem]{Lemma}
\newtheorem{proposition}[theorem]{Proposition}

\newtheorem{claim}[theorem]{Claim}

\theoremstyle{definition}

\newtheorem{example}[theorem]{Example}
\newtheorem{remark}[theorem]{Remark}

\newtheorem{definition}[theorem]{Definition}

\title[MAXIMAL PAGE CROSSING NUMBER OF CLOSED ORIENTABLE LEGENDRIAN SURFACES]{MAXIMAL PAGE CROSSING NUMBERS OF LEGENDRIAN SURFACES IN CLOSED CONTACT 5-MANIFOLDS}

\author{M. Firat Arikan}
\address{Dept. of Mathematics, Middle East Technical University, Ankara, TURKEY}
\email{farikan@metu.edu.tr}

\author{Ozlem Ersen}
\address{Dept. of Mathematics, Middle East Technical University, Ankara, TURKEY}
\email{ozlem.ersen@metu.edu.tr}

\date{\today}

\begin{document}

\begin{abstract}
We introduce a new Legendrian isotopy invariant for any closed orientable Legendrian surface $L$ embedded in a closed contact $5$-manifold $(M, \xi)$ which admits an ``admissable'' open book $(B, f)$ (supporting $\xi$) for $L$. We show that to any such $L$ and a fixed page $X$, one can assign an integer $M\mathcal{P}_{X}(L)$, called ``Relative Maximal Page Crossing Number of $L$ with respect to $X$", which is invariant under Legendrian isotopies of $L$. We also show that one can extend this to a page-free invariant, i.e.,  one can assign an integer $M\mathcal{P}_{(B,f)}(L)$, called ``Absolute Maximal Page Crossing Number of $L$ with respect to $(B, f)$", which is invariant under Legendrian isotopies of $L$. In particular, this new invariant distinguishes Legendrian surfaces in the standard five-sphere which can not be distinguished by Thurston-Bennequin invariant.
\end{abstract}

\maketitle


\section{Introduction}

Let $(M^5, \xi=\textrm{Ker}(\alpha))$ be a closed contact $5$-manifold where $\alpha$ is a (global) contact form with the Reeb vector field $R$ that is compatible with an open book $(B, f)$ on $M$.  Consider the associated abstract open book $OB(X, h)$ where $X$ is the page, $B=\partial X$ is the binding, and $h$ is the monodromy. (See the next section for definitions.) Thus, $(B^3, \xi|_{B}=\textrm{Ker}(\alpha|_{B}))$ is the convex boundary of each symplectic page $(X^4,d\alpha|_{B} )$, and so it is a $3$-dimensional tight contact (sub)manifold. ($\alpha|_{B}$ is a contact form on $B$.) Let $L$ be a closed orientable Legendrian surface of $(M, \xi)$, and so there is a Legendrian embedding $\phi :\Sigma \hookrightarrow \big (M, \xi\big)$ such that $\phi(\Sigma)=L$ where $\Sigma$ is a $2$-dimensional surface which determines the topological type of $L$.  
For the invariants that we will define, one needs that $L$ and $B$ intersect transversely, and $X$ is simply-connected and Weinstein. To this end, we define the following class of supporting open books:\\

\begin{definition}
	Let $L$ be a compact, oriented, Legendrian submanifold of a closed contact $5$-manifold $(M, \xi)$. An open book $(B, f)$ on $M$ supporting $\xi$ is called an \textbf{\textit{admissable open book for} $L$} if it has simply-connected Weinstein pages and $L$ intersects $B$ transversely.
\end{definition}

One can always find an open book with Weinstein pages whose binding $B$ is transverse to a given $L\subset (M, \xi)$. This can be seen from a combination of results: Theorem \ref{thm:existence_OB} and Lemma \ref{lem:spider}. Suppose $(B, f)$ is such an admissable open book, and $X$ is any  fixed page of $(B, f)$. Then\\

\begin{theorem} \label{thm:relative_page_crossing}
One can associate an integer $M\mathcal{P}_{X}(L)$, called ``Relative Maximal Page Crossing Number of $L$ with respect to $X$", which is invariant under Legendrian isotopies of $L$.\\
\end{theorem}

The definition of $M\mathcal{P}_{X}(L)$ are given in Section \ref{sec:Relative_Page_Crossing}, and based on the link of (transverse) intersection of $L$ with the double $D(X)$ of the page $X$ that we fix. The proof of Theorem \ref{thm:relative_page_crossing} will be presented in Section \ref{sec:well-definedness_isotopies}. Using the relative version and putting a further \textbf{\textit{essentially intersecting}} condition on $(B, f)$ (see Definition \ref{def:Page_crossing} and Section \ref{sec:well-definedness_absolute} for the definition), one can also define a number which is independent of pages of the open book at hand. Namely, under the above assumptions, we prove:\\

\begin{theorem} \label{thm:absolute_page_crossing}
One can associate an integer $M\mathcal{P}_{(B,f)}(L)$, called ``Abolute Maximal Page Crossing Number of $L$ with respect to $(B,f)$", which is invariant under Legendrian isotopies of $L$.\\
\end{theorem}

We define $M\mathcal{P}_{(B,f)}(L)$ and prove Theorem \ref{thm:absolute_page_crossing} in Section \ref{sec:well-definedness_absolute}.\\

Any orientable Legendrian submanifold in any contact manifold comes with a canonical contact framing, called \textit{Thurston-Bennequin framing}. More precisely, if $L^n \subset (M^{2n+1},\xi)$ is an orientable Legendrian submanifold, then its contact framing is determined by a smooth vector field which is every transverse to $\xi|_L$. If we further assume that $L$ is null-homologous (i.e., if $L=\partial C$ for some $(n+1)$-chain $C \subset M$), then we can compare the contact framing on $L$ with the one determined by $C$, and so one can identify it with an integer $tb(L)$ called \textit{Thurston-Bennequin number} of $L$. (See the next section for more details.) In the past two decades new Legendrian isotopy invariants have been defined and studied (see for instance, \cite{EES1}, \cite{EES2}, \cite{EES5}) due to insufficiency of $tb(L)$ in distinguishing non-isotopic Legendrian submanifolds in certain cases. Most of these new invariants are based on differential graded algebras and very difficult to compute.\\

Returning back to dimension five, it has been known (see \cite{EES5}, for instance) that $tb(L)$  can not distinguish Legendrian surfaces in the standard contact $\mathbb{R}^{5}$ or $\mathbb{S}^5$ which are smoothly (but not Legendrian) isotopic. The reason for this is that $tb(L)$ coincides with a topological invariant for these cases, i.e., it does not carry any information about the Legendrian embedding of $L$ into  $\mathbb{R}^{5}$ or $\mathbb{S}^5$. On the other hand, the invariants $M\mathcal{P}_X(L)$, $M\mathcal{P}_{(B,f)}(L)$ can distinguish such Legendrian surfaces by means of classical computations relatively more visual and simpler than those used in computing other Legendrian isotopy invariants.  A concrete example is given in Section \ref{sec:Example}.


\section{Preliminaries} \label{sec:Preliminaries}



\subsection{ Contact and Symplectic Manifolds}\label{subsec_contact_manifolds}

Let us start with defining contact structures on (necessarily) odd and symplectic structures on (necessarily) even dimensional manifolds. More discussions and details about them can be found, for instance, in \cite{E3}, \cite{Ge} and \cite{MS8}.

\begin{definition}
	A pair $(M^{2n+1},\xi)$ is called a \textbf{contact manifold} where $M$ is a  smooth manifold and $\xi\subset TM$ is a totally non-integrable $2n$-plane field distrubution on $M$, that is, locally, $\xi$ is the kernel of a $1$-form $\alpha$ with the property $\alpha\wedge(d\alpha)^n\neq0$. The \textbf{contact structure} $\xi$ is said to be \textbf{co-oriented} if it is the kernel of a globally defined $1$-form $\alpha$ with the above property. Such $\alpha$ is called a \textbf{contact form} on $M$. Corresponding to a contact form $\alpha \in \Omega^{1}(M)$, the \textbf{\textit{Reeb vector field}} is the vector field $R_{\alpha}$ uniquely defined by the equations
	$$d\alpha(R_{\alpha}, -)=0, \quad \alpha(R_{\alpha})=1.$$

Finally, a vector field $Z$ on a contact manifold is said to be a \textbf{\textit{contact vector field}} if it satisfies 
 \begin{center}
 	$\mathcal{L}_{Z} \alpha=f \alpha$
 \end{center}
for some function $f: M\rightarrow \mathbb{R}$.
That is, the flow of $Z$ preserves the contact distrubution $\xi$.	
\end{definition}

Through out the paper, all contact manifolds will be assumed to be co-oriented. Any contact manifold $(M,\xi)$ is necessarily orientable as $\alpha\wedge(d\alpha)^n\neq0$. Once a contact form $\alpha$ for $\xi$ is fixed, $M$ is assumed to be oriented by the volume form $\alpha\wedge(d\alpha)^n$.


\begin{example} Consider the standard coordinates $(x_{1}, y_{1}, x_{2}, y_{2}, z)$ in $\mathbb R^5$, and also the 1-form $\alpha_{0}=dz+x_{1}dy_{1}+x_{2}dy_{2}$. Since $\alpha_{0}\wedge(d\alpha_{0})^2=(dz+x_{1}dy_{1}+x_{2}dy_{2})\wedge(dx_{1}\wedge dy_{1}+dx_{2}\wedge dy_{2})^2=2dz\wedge dx_{1}\wedge dy_{1}\wedge dx_{2}\wedge dy_{2}\neq0$, $\xi_{0}=\textrm{Ker}(\alpha_{0})$ is a contact structure on $\mathbb R^5$. The $4$-plane distribution $\xi_{0}$ is called the \textbf{\textit{standard contact structure}} on $\mathbb R^5$. Note that the Reeb vector field of $\alpha_{0}$ is $\partial_{z}$.\end{example}

\begin{example} Let $\mathbb{S}^5$ be the unit $5$-sphere in $\mathbb R^6$ with usual coordinates $(x_{1}, y_{1}, x_{2}, y_{2}, x_{3}, y_{3})$. Consider the $1$-form $\alpha_{st} =x_{1}dy_{1}-y_{1}dx_{1}+x_{2}dy_{2}-y_{2}dx_{2}+x_{3}dy_{3}-y_{3}dx_{3}$, restricted to $\mathbb{S}^5$. This contact form defines the \textbf{\textit{standard contact structure}} $\xi_{st}=\textrm{Ker}(\alpha_{st})$  on $\mathbb{S}^5$.\end{example}

The following definitions describe the equivalence of contact structures and forms.

\begin{definition}
	A diffeomorphism $\psi:(M_{1}, \xi_{1}=\textrm{Ker}(\alpha_{1})) \rightarrow (M_{2}, \xi_{2}=\textrm{Ker}(\alpha_{2})) $ between two contact manifolds is called \textbf{\textit{contactomorphism}} if its derivative map $\psi_*: TM_{1}\rightarrow TM_{2}$ takes the contact structure $\xi_{1}$ to the contact structure $\xi_{2}$ on $M_{2}$, i.e. if there is a function $\lambda : M_{1}\rightarrow \mathbb{R}\setminus \{0\}$ with $\psi^{*} \alpha_{2}=\lambda\alpha_{1}$. Two contact manifolds $(M_{1}, \xi_{1})$ and $(M_{2}, \xi_{2})$ are said to be \textbf{\textit{contactomorphic}} if there exists a contactomorphism between them.
\end{definition}

\begin{definition}
Two contact structures $\xi_{1}$ and $\xi_{2}$ on a manifold $M$ are \textbf{\textit{isotopic}} if there is a contactomorphism $\psi:(M, \xi_{1})\rightarrow (M, \xi_{2})$ such that $\psi$ is isotopic to the identity. On the other hand, two contact structures $\xi_{1}$ and $\xi_{2}$ on $M$ are called \textbf{\textit{homotopic}} if they are homotopic as hyper-plane distributions.
\end{definition}

We note that two different contact structures can be homotopic but not isotopic. Hence, classification of contact structures is made upto isotopy. 
Following local results hold in any odd dimension, but we state them here only for dimension five, and use them later in the paper.

\begin{theorem} 
For any point $p \in \mathbb{S}^{5}$, $(\mathbb{S}^{5}\backslash \{p\}, \xi_{st})$ and $(\mathbb R^{5}, \xi_{0})$ are contactomorphic. 
\end{theorem}

\begin{theorem} (Darboux's Theorem)  Let $(M,\xi)$ be any contact $5$-manifold, $p\in M$ any point. Then there is a neighborhood $U$ of $p$ in $M$ such that $(U, \xi\mid_{U})$ is contactomorphic to $(\mathbb R^{5}, \xi_{0})$. 
\end{theorem}

A neighborhood $U$ as in the above theorem is said to be a \textbf{\textit{Darboux ball}}. In dimension three, contact structures arise in two different types. This difference plays an important role in our upcoming discussions.

\begin{definition}
	If there is an embedded disk $D$ in $(M^3, \xi)$ such that $T_{p}(\partial D) \subseteq \xi_{p}$ at every point $p \in \partial D$, then $\xi$ is called an \textbf{\textit{overtwisted}} contact structure. Such a disk $D$ is called an \textbf{\textit{overtwisted disk}}. Otherwise, $\xi$ is called a \textbf{\textit{tight}} contact structure.
\end{definition}

Next, let us recall symplectic and almost complex structures:

\begin{definition}
	A \textbf{\textit{symplectic structure}} on a manifold $X$ is a closed $2$-form $\omega  \in \Omega^{2}(M)$ (i.e., $d\omega=0$) which is nondegenerate at every $p \in X$ (i.e.,  $\forall v \in T_{p}X, v\neq 0, \exists u \in T_{p}X$ such that $\omega_{p}(v, u)\neq 0$). The pair $(X, \omega)$ is called a \textbf{\textit{symplectic manifold}}.
\end{definition}

Note that any symplectic manifold $(X, \omega)$ is necessarily even dimensional (say $2n$) and oriented. In fact, the nondegeneracy condition above is equivalent to $\omega^{n} \neq 0$. Thus, the top (volume) form $\omega^{n}$ defines the canonical (symplectic) orientation on $X$.


\begin{definition}
	A \textbf{\textit{symplectomorphism}} between two symplectic manifolds $(X_1, \omega_1), (X_2, \omega_2)$ is a diffeomorphism $\psi :X_1\rightarrow X_2$ with the property $\psi^{*}\omega_2=\omega_1$. 
	\end{definition}
	
 \begin{definition}
 	An \textbf{\textit{almost complex structure}} $J$ on a smooth $2n$-manifold $X$ is an assignment of complex structures $J_{p}$ on the tangent spaces $T_{p}X$ which depends smoothly on $p$. The pair $(X, J)$ is called an \textbf{\textit{almost complex manifold}}.
 In other words, an almost complex structure on $X$ is a $(1, 1)$-tensor field $J:TX\rightarrow TX$ such that $J^{2}=-Id$.
 \end{definition}

\begin{definition}
	An almost complex structure $J$ on $X$ is \textbf{\textit{compatible with}} a symplectic structure $\omega$ on $X$ if $\omega(u, v)=\omega(Ju, Jv)$ for all $u, v \in TX$ (i.e., $J$ preserves $\omega$), and $\omega(u, Ju)> 0$ for all nonzero $u\in TX$ (called the \textbf{\textit{taming condition}}).
\end{definition}

\begin{theorem} 
	The space of all compatible almost complex structures on $(X, \omega)$ is contractible, and hence non-empty.
\end{theorem}

The above theorem is due to Gromov \cite{G5} (for a proof see also \cite{Ge} or \cite{MS8}) and provides a very useful connection between symplectic and (almost) complex geometry.


\subsection{Liouville, Weinstein and Stein Manifolds} \label{subsec_Liou_Wein_Stein}

 Now we recall special families of symplectic manifolds in which we are interested. More details about definitions and facts given below can be found in \cite{C2} and \cite{MS8}.

\begin{definition}
	A \textbf{\textit{Liouville cobordism}} is a  symplectic cobordism $(X, \omega)$ from $\partial_- X=M_{-}$ to $\partial_+ X= M_{+}$ with a Liouville structure. A \textbf{\textit{Liouville structure}} means that there is a $1$-form $\alpha$ on $X$ such that $\omega=d\alpha$ and the $\omega$-dual vector field $Z$ of $\alpha$ is a \textbf{\textit{Liouville vector field}} for $\omega$ (i.e., $\mathcal{L}_{Z} \omega=\omega$) transversely pointing inward (resp. outward) along the boundary component $\partial_-X$ (resp. $\partial_+X$). A Liouville cobordism with $\partial_- X=\emptyset$ is called a \emph{\textbf{Liouville domain}}.\\
	When $X$ is an open manifold, if we assume that the flow of $Z$ exists for all times and there exists an exhaustion $X=\bigcup_{k=1}^{\infty}X^{k}$ by compact domains $X^{k}\subset X$ such that each $(X^{k}, \alpha|_{X^k})$ is a Liouville domain with convex boundary $(\partial X^{k}, \alpha|_{\partial X^k})$ for all $k\geq 1$, then $(X, \alpha)$ is called a \textbf{\textit{Liouville manifold}}. Since $\omega$ and $Z$ uniquely determine $\alpha$ (namely, $\alpha=\iota_Z\omega$), one can also use the notation $(X, \omega, Z)$ for Liouville cobordisms/domains/manifolds.
\end{definition}


By putting more conditions on Liouville manifolds, one can consider the class of Weinstein/Stein manifolds. 
In order to define them, we need some preliminary definitions:

\begin{definition}
(i) A vector field $Z$ on a smooth manifold $X$ is said to \emph{\textbf{gradient-like}} for a smooth function $\phi: X \to \R$ if $Z \cdot \phi=\mathcal{L}_{Z} \phi >0$ away from the critical point of $\phi$.\\
(ii) A real-valued function is said to be \emph{\textbf{exhausting}} if it is proper and bounded from below.\\
(iii) An exhausting function $\phi: X\to \R$ on a symplectic manifold $(X,\omega)$ is said to be $\omega$-\emph{\textbf{convex}} if there exists a complete Liouville vector field $Z$ which is gradient-like for $\phi$.\\
(iv) Suppose that $(X, J)$ is an almost complex manifold. Then a smooth map $\phi:X \rightarrow \mathbb{R}$ is said to be $J$-\textbf{\textit{convex}} if $\omega_{\phi}:=-d(d \phi\circ J)$ is nondegenerate (i.e., $\omega_{\phi}(v, Jv)>0$ for all $v \neq 0$), and so symplectic.
\end{definition}

\begin{definition} \label{def:Weinstein_manifold}
A \emph{\textbf{Weinstein manifold}} $(X,\omega,Z,\phi)$ is a symplectic manifold $(X,\omega)$ which admits a $\omega$-convex Morse function $\phi:X \to \R$ whose complete gradient-like Liouville vector field is $Z$. The triple $(\omega,Z,\phi)$ is called a \emph{\textbf{Weinstein structure}} on $X$. A \emph{\textbf{Weinstein cobordism}} $(X,\omega,Z,\phi)$ is a Liouville cobordism $(X,\omega,Z)$ whose Liouville vector field $Z$ is gradient-like for a Morse function $\phi:X \to \R$ which is constant on the boundary $\partial X$. A Weinstein cobordism with $\partial_- X=\emptyset$ is called a \emph{\textbf{Weinstein domain}}.
\end{definition}

Any Weinstein manifold $(X,\omega,Z,\phi)$ can be exhausted by Weinstein domains $$X_k = \{\phi^{-1}(-\infty,d_k] \} \subset X$$ where $\{d_k\}$ is an increasing sequence of regular values of $\phi$, and therefore, any Weinstein manifold is a Liouville manifold. In particular, any Weinstein domain is a Liouville domain. Also note that any Weinstein domain $(X,\omega,Z,\phi)$ has the convex boundary $(\partial X, \textrm{Ker}((\iota_{Z}\omega)|_{\,\partial X}))$.

The following topological characterization of  Weinstein domains will be important for us.

\begin{theorem} [\cite{W1}, see also Lemma 11.13 in \cite{C2}] \label{thm:Top_Charac_of_Weinstein_manifolds}
Any Weinstein domain of dimension $2n$ admits a handle decomposition whose handles have indices at most $n$.
\end{theorem}

Originally, Stein manifolds are defined as the class of manifolds which can be holomorphically embedded into some complex space $\mathbb{C}^N$ for $N$ large enough, and hence they are complex manifolds. In terms of the structure of the present paper, they can be defined as follows:

\begin{definition}
A \textbf{\textit{Stein manifold}}	is a triple $(X, J, \phi)$ where $J$ is an almost complex structure on $X$ and $\phi:X \rightarrow \mathbb{R}$ is an exhausting $J$-convex Morse function which is also $\omega_{\phi}$-convex. A \emph{\textbf{Stein cobordism}} $(X,J,\phi)$ is a Weinstein cobordism $(X,\omega_{\phi},Z,\phi)$. A Stein cobordism with $\partial_- X=\emptyset$ is called a \emph{\textbf{Stein domain}}. 
\end{definition}

It is not hard to observe that there is an underlying a Weinstein structure for any given Stein structure. Indeed, it has been shown that the converse is also true: 

\begin{theorem} [\cite{C2}] \label{thm:Weinstein_implies_Stein}
Any Weinstein structure on a manifold $X$ can be deformed to another one which is the underlying Weinstein structure of some Stein structure on $X$.
\end{theorem}

\begin{definition}
A contact manifold $(M, \xi)$ is called \textbf{\textit{Stein fillable}} (or \textbf{\textit{holomorphically fillable}} if there is a Stein domain $(X, J, \phi)$ such that $\partial X=M$ and $\xi=\textrm{Ker}(-(d \phi\circ J)|M)$.
\end{definition}


\begin{theorem} [\cite{EG}] \label{thm:Stein_fillable_tight}
Any Stein fillable contact $3$-manifold is tight.	
\end{theorem}


\subsection{Open Book Decompositions}\label{subsec_open_books}

Open book decompositions are topological structures and they have a strong relationship with contact structures. We refer the reader to \cite{Ge} and \cite{Et} for more details.
	
\begin{definition}
An (\textbf{\textit{embedded}} or \textbf{\textit{non-abstract}}) \textbf{\textit{open book}} (\textbf{\textit{decomposition}}) of a closed $(2n+1)$-manifold $M$ is determined by a pair $(B, f)$ where $B$ is a codimension $2$ submanifold with trivial normal bundle and $f : M \setminus B\rightarrow S^{1}$ is a fiber bundle projection such that the normal bundle has a trivialization $B \times D^{2}$, where the angle coordinate on the disk agrees with the fibration map $f$. The $(2n-1)$-manifold $B$ is called the \textbf{\textit{binding}} and for any $t_{0} \in S^{1}$, the $2n$-manifold $X=f^{-1}(t_{0})$ (a fiber of $f$) is called a \textbf{\textit{page}} of the open book.
\end{definition}

An alternative definition of an open book decomposition can be given as follows:

\begin{definition}
	An open book $(B, f)$ determines an \textbf{\textit{abstract open book}} $(\bar{X},h)$ where $\bar{X}$ denotes the closure of a page $X$ in $M$, and $h:\bar{X} \to \bar{X}$ is the self-diffeomorphism (which is identity near the binding $B=\partial \bar{X}$) defined by the time-one map of the flow lines along the $S^1$-direction. The map $h$ is called the \textbf{\textit{monodromy}} of the open book decomposition.
\end{definition}

In fact, the two notions of open book decomposition are closely related. The difference is that when discussing open books (non-abstract), we can discuss the binding and pages up to isotopy in $M$, whereas when discussing abstract open books we can only discuss them up to diffeomorphism.

	
The following definition is due to Giroux:

\begin{definition}[\cite{G2}] \label{def:compatibility}
	A contact structure $\xi$ on a closed $(2n+1)$-manifold $M$ is said to be \textbf{\textit{supported by}} (or \textbf{\textit{carried by}}, or \textbf{\textit{compatible with}}) an open book $(B, f)$ on $M$ if there exists a contact form $\alpha$ for $\xi$ such that 
	\begin{itemize}
		\item[(i)] $(B, \textrm{Ker}(\alpha|_{TB}))$ is a contact $(2n-1)$-manifold,
		\item[(ii)] for any $t\in S^1$, the page $(X=f^{-1}(t),d\alpha|_{TX})$ is a symplectic $2n$-manifold, and
		\item[(iii)] if $\bar{X}$ denotes the closure of a page $X$ in $M$, then the orientation on $B=\partial \bar{X}$ induced by its contact form $\alpha |_{TB}$ coincides with its orientation as the convex boundary of $(\bar{X}, d\alpha|_{TX})$.
	\end{itemize}
\end{definition}



\begin{theorem}[\cite{G2}] \label{thm:existence_OB}
Every contact structure on a closed manifold is compatible with some open book decomposition with Weinstein (and so Stein) pages.
\end{theorem}


We now explain how the page and the monodromy of an open book changes under a certain process called \textit{stabilization}.

\begin{definition}[\cite{G2}]
	Let $D^{n} \subset X^{2n}$ be an $n$-dimensional disc embedded into the $2n$-dimensional page of an open book $(X,h)$ of an odd dimensional manifold $M$ such that $D^{n}$ meets $\partial X$ transversely and exactly in its boundary $\partial D^{n}$ and such that the normal bundle of $\partial D^{n}$ in $\partial X$ is trivial. Attach an $n$-handle $H$ to $X$ along  $\partial D^{n}$ in such a way that the normal bundle of the sphere $S^{n}=D^{n}\cup core(H)$ is isomorphic to $T^{*}S^{n}$. Then the open book $(X \cup H, h\circ\tau)$ is called a \textbf{\textit{positive stabilization}} of $(X,h)$, where $\tau$ denotes a right-handed Dehn twist along the sphere $S^{n}$. Similarly, one can also define \textbf{\textit{negative stabilization}} using left-handed Dehn twist $\tau^{-1}$ instead of right-handed one. 
\end{definition}

\begin{remark}\label{rmk:stabilization}
We note that the original open book $(X,h)$ and the stabilized open book $(X \cup H, h\circ\tau)$ give rise (up to diffeomorphism) to the same manifold $M$. Indeed, the sphere $\partial D^{n} \subset \partial X= B \subset (X,h) $ is a sphere with trivial normal bundle in $M$, since the binding $B$ has trivial normal bundle by definition. Attaching handles to each page is equivalent to a surgery along $\partial D^{n}$. The manifold $M'$ obtained by that surgery carries the open book structure $(X \cup H,h)$. Performing the Dehn twist $\tau$ (or $\tau^{-1}$) along $S^{n}$ is equivalent to a surgery cancelling the one corresponding to the handle attachment.
\end{remark}

Although contact structures are purely geometric objects while open book decompositions are purely topological, Giroux found a very useful relation between them as stated below:

\begin{theorem}\cite{G2} \label{thm:Giroux_Corresp}
	Let $M$ be a closed  $(2n+1)$-manifold. Then there is one to one correspondence between co-oriented contact structures on $M$ up to isotopy and open book decompositions of $M$ with Weinstein (and so Stein) pages up to positive stabilization.
\end{theorem}

This correspondence between co-oriented contact structures and open book decompositions is called the \textbf{\textit{Giroux Correspondence}}. 




\subsection{Legendrian Submanifolds and Thurston-Bennequin Invariant} \label{subsec_Legendrian_submanifolds}

Legendrain submanifolds are the most interesting ones in contact geometry. Although they are defined in any odd dimensions, we restrict our attention mostly to dimension five and three. The non-integrability condition of contact $5$-manifolds ensures that there is no submanifold of dimension greater than or equal to $3$ which is tangent to the contact distribution. However, we can find $2$-dimensional submanifolds whose tangent spaces lie inside the contact field. 

\begin{definition}
	Let $(M^{5}, \xi)$ be a contact manifold. A submanifold $L$ of $(M^{5}, \xi)$ is called an \textbf{\textit{isotropic submanifold}} if $T_{p}L \subset \xi_{p}$ for all points $p \in L$. An isotropic submanifold of dimension two (an \textit{isotropic surface}) is called a \textbf{\textit{Legendrian submanifold}} (\textbf{\textit{surface}}). A \textbf{\textit{Legendrian embedding}} is an embedding $\phi:\Sigma^2 \hookrightarrow (M^{5},\xi)$ of a smooth manifold $\Sigma^2$ such that the image $L^2=\phi(\Sigma^2)$ is an embedded Legendrian surface. A smooth 1-parameter family of embedded Legendrian surfaces is called a \textbf{\textit{Legendrian isotopy}}. Two Legendrian surfaces $L$, $L'$ are called \textbf{\textit{Legendrian isotopic}} if there is a smooth $1$-parameter family $L_{t}$, $t \in [0, 1]$, of embedded Legendrian surfaces such that $L_{0}=L$ and $L_{1}=L'$. Equivalently, a Legendrian isotopy is a smooth 1-parameter family $\phi_t:\Sigma^2 \hookrightarrow (M^{5},\xi)$ of Legendrian embeddings.
\end{definition}

Legendrian knots inside a contact $3$-manifold are the simplest example of Legendrian submanifolds. Indeed all the terms in the above definitions can be restated for Legendrian knots as well. In particular, two Legendrian knots are equivalent if they are isotopic via a family of Legendrian knots. Nullhomologous Legendrian knots of the same topological knot type can be distinguished by their Thurston Bennequin and rotation numbers, which are Legendrian isotopy invariants. Thurston-Bennequin invariant (see below for its definition) was originally defined by Bennequin \cite{B3}  and independently, Thurston when $n=1$, and generalized to higher dimensions by Tabachnikov \cite{T8}.\\

Let $L^n$ be an orientable connected nullhomologous Legendrian submanifold in a co-oriented contact $(2n+1)$-manifold $(M, \xi)$. Pick an orientation on $L$. Let $R$ be a Reeb vector field for $\xi$. Push $L$ slightly off of itself along $R$ to get another oriented submanifold $L'$ (a Legendrian copy of $L$ with the push-forward orientation) disjoint from $L$. The \textbf{\textit{Thurston-Bennequin invariant}} (\textbf{\textit{number}}) of $L$  is the linking number of $L$ and $L'$, that is, we have
\begin{center}
	$tb(L):=lk (L , L') \in\mathbb{Z}$
\end{center}
where $lk$ denotes the linking number. For the linking number, take any $(n+1)$-chain $C$ in $M$ such that $\partial C=L$. Then $lk (L , L')$ equals the algebraic intersection number of $C$ with $L'$. Intuitively, the Thurston-Bennequin invariant (number) of $L$ measures the twisting of $\xi$ around $L$. We note that $tb(L)$ is independent of the chosen orientation of $L$, and it is a Legendrian isotopy invariant in any odd dimension.\\

For a Legendrian knot $K$ in $(S^3,\xi_{st})$ (or equivalently in $(\mathbb{R}^3,\xi_{0}=\textrm{Ker}(dz+xdy))$), Thurston Bennequin number can be computed as follows: Pick an orientation on $K$. Then
\begin{center}
	$tb(K)=w(K)-\dfrac{1}{2}c(K)$
\end{center}
where $w(K)$ is the \emph{\textbf{writhe}} of $K$; i.e., the sum of the signs of the crossings of $K$ determined as in Figure \ref{fig:26}, and $c(K)$ is the number of \textit{cusps} in the \textbf{\textit{front projection}} of $K$ (the projection of $K$ onto the $yz$-plane). Here \emph{\textbf{cusps}} are the singular points in the front projection of $K$. (Note that $w(K)$ is independent of the chosen orientation of $K$.)

\begin{figure}[h]
	\centering	
	\includegraphics{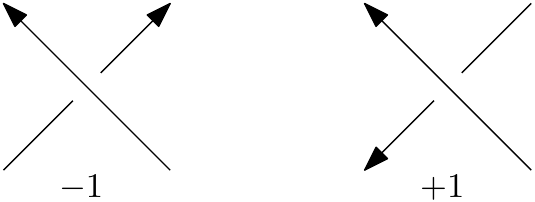}
	\caption{Positive and negative crossings.}
	\label{fig:26}
\end{figure}

\begin{definition}
	For an oriented Legendrian knot $K$ in $(\mathbb{R}^3,\xi_{0}=\textrm{Ker}(dz+xdy))$ (or equivalently in $(S^3,\xi_{st})$), its \textbf{\textit{rotation number}}  $rot(K)$ is defined as
	\begin{center}
		$rot(K)=\dfrac{1}{2}(D-U)$
	\end{center}
where $D$ (resp. $U$) denotes the number of down (resp. up) cusps in the front projection of $K$.
	\end{definition}
	
For a fixed topological knot type, different Legendrian representatives might have different Thurston Bennequin and rotation numbers. By adding more cusps, one can obtain new Legendrian representatives realizing any pregiven integer as a rotation number. However, although Thurston Bennequin number can be made arbitrarly small, it is not possible to increase forever. The following result (due to Bennequin and Eliashberg) provides an upper bound for the Thurston Bennequin number for nullhomologous Legendrian knots in tight contact $3$-manifolds. 

\begin{theorem} [\textbf{\textit{Bennequin inequality}}] \label{thm:tb_bound}
Let $K$ be a Legendrian knot in a tight $3$-manifold $(M,\xi)$ which bounds a surface $\Sigma \subset M$. Then
		$$tb(K)+ \mid rot(K) \mid \leq -\chi(\Sigma)$$
where $\chi(\Sigma)$ denotes the Euler characteristic of $\Sigma$. 
\end{theorem}
 

\subsection{Handle decompositions of Stein surfaces} \label{subsec_Stein}

Let us first recall handlebodies. A copy of $D^{k}\times D^{n-k}$ that is attached to the boundary of an $n$-manifold along its \textbf{\textit{attaching region}} $\partial D^{k} \times D^{n-k}$ is called an $n$-\textbf{\textit{dimensional handle of index}} $k$ (or simply a $k$-\textbf{\textit{handle}}). Starting from a $0$-handle, a manifold obtained from attaching (finitely many) such $k$-handles ($k=0,1,...,n$)  is called a  \textbf{\textit{(smooth or topological) handlebody (of finite type)}}. For the smooth case, we glue each handle by a smooth embedding of its attaching region and smoothen the resulting corners. This construction results in a real-valued Morse function on the resulting manifold. Conversely, starting from a real-valued Morse function on a manifold $X$, one can obtain its handlebody description which is also referred to as a \textbf{\textit{handle decomposition}} of $X$. In the category of smooth $4$-manifolds, a handle decomposition of a manifold $X$ describes not only the topology but also a smooth structure on $X$. (The details can be found in \cite{GS}.)\\


The phrase ``a Stein surface'' will refer to a Stein domain of real dimension $4$. Pictures of handlebody diagrams of Stein surfaces (\textbf{\textit{Stein diagrams}} for short) were studied extensively by Gompf \cite{G8}. He gave description of $1$-handles in the setting of Stein surfaces and a standard form for Legendrian links in \#$nS^{1}\times S^{2}$($=$ Boundary of the $0$-handle $\cup$ $n$ $1$-handles). From this description, one can define and compute Thurston-Bennequin invariant as explained below.

\begin{definition} [\cite{G8}]
 A \textbf{\textit{Legendrian link diagram} } in the standard form, as in Figure \ref{fig:figure15}, is defined by the following way:\\
\noindent $\bullet$ $n$ 1-handles, showed by $n$ pairs of horizontal balls.\\
$\bullet $ A collection of $n$ horizontal distinguished segments coresponding to each pair of ball. \\
$\bullet$ A front projection of a generic \textbf{\textit{Legendrian tangle}} (i.e., disjoint union of Legendrian knots and arcs) with endpoints touching the segments.
\end{definition}

\begin{figure}[h!]
	\centering
	\includegraphics{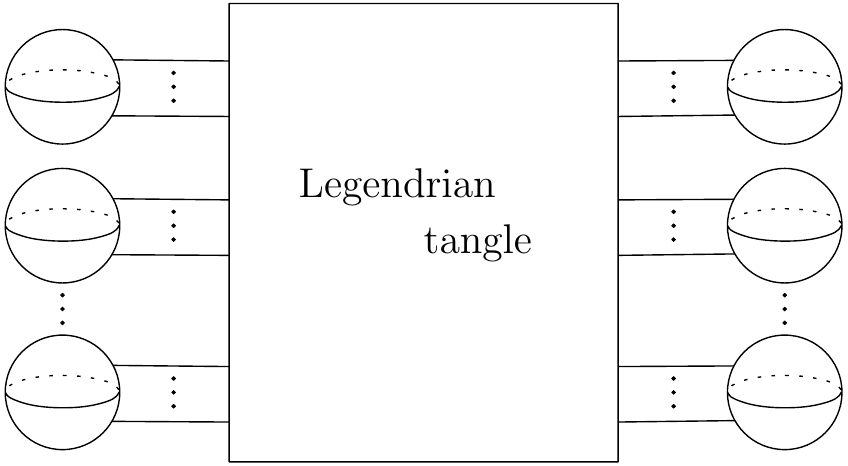}
	\caption{A Legendrian link diagram.}
	\label{fig:figure15}
\end{figure}

Similar to how it is defined for Legendrian knots in the standard contact three-space, the Thurston-Bennequin number of a Legendrian knot $K$ in a boundary of a Stein surface can be defined as
\begin{center}
	$tb(K)=w(K)-\dfrac{1}{2}c(K)$
\end{center}
with  the help of a Legendrian tangle (see \cite{G8}). The following result will be used later:

\begin{theorem} [\cite{E0}, \cite{G8}] \label{thm:Stein_handlebody_condition}
An oriented, compact, connected $4$-manifold $X$ is a Stein surface if and only if it has a handlebody diagram which formed by a Legendrian link diagram such that $2$-handles attached to link components $L_{i}$'s with framing $tb(L_{i})-1$.
\end{theorem}


\section{Relative Page Crossing Number} \label{sec:Relative_Page_Crossing}

Let us start with showing that one can always find an open book whose binding intersects a given Legendrian surface transversely. For similar arguments, we refer the reader to \cite{AF19}.

\begin{lemma}[\textit{Spider Lemma}] \label{lem:spider}
	Let $(M, \xi)$ be a closed contact $5$-manifold and $(B, f)$ an open book on $M$ supporting $\xi$ with Weinstein pages. Also let $L$ be a closed orientable Legendrian surface in $M$. Then there exists an isotopy $(B_t, f_t), t\in[0,1]$ of open books all of which supporting $\xi$ such that $(B_0, f_0)=(B,f)$, $(B_1, f_1)=(B',f')$, and $L$ intersects $B'$ transversely.
\end{lemma} 

\begin{proof}
	If $L$ and $B$ transversely intersect, then there is nothing to prove. If they don't intersect transversely, then consider a neighborhood of $B$ in $M$ which can be identified with $B\times D^2$. Nearby generic $B' \subset B\times D^2$ (which is a copy of $B$) intersects $L$ transversely. Then we can isotope $B$ to $B'$ (and accordingly the pages of the open book $(B,f)$) using the flow of an appropriate contact vector field compactly supported in $B\times D^2$. (See Figure \ref{fig:30}.) So, we obtain a family of open books $\{(B_t, f_t)\}$ for $M$ such that $(B_0, f_0)=(B,f)$, $(B_1, f_1)=(B',f')$. Finally, we note that at any time $t\in[0,1]$ compatibility conditions in Definition \ref{def:compatibility} are satisfied by the open book $(B_t, f_t)$ since the isotopy is based on a contact vector field.
\end{proof}

\begin{figure}[h!] 
		\centering	\includegraphics[width=.9\textwidth,height=.3\textheight]{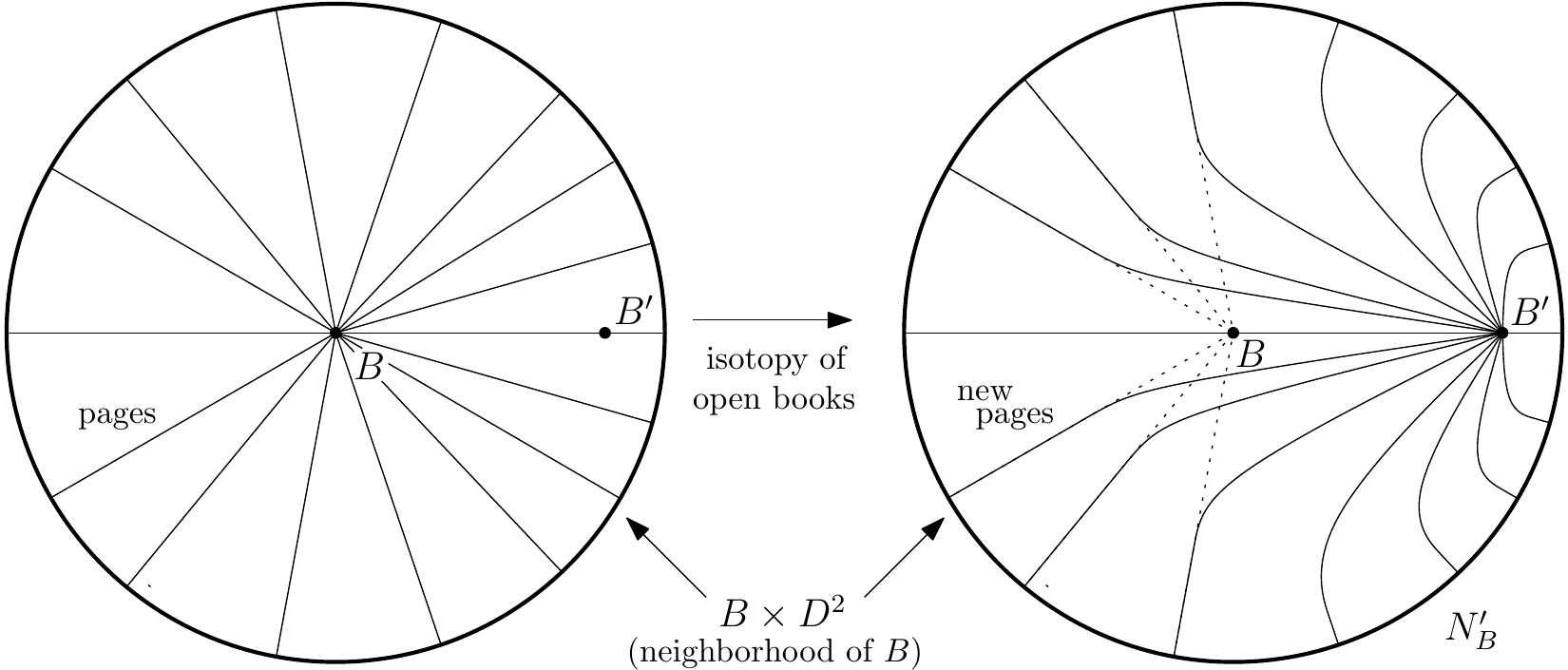}
		\caption{ Isotoping a given open book so that the new binding $B'$ intersects $L$ transversally. }
		\label{fig:30}
	\end{figure}

Assuming Spider Lemma have been already applied, we may start with a supporting open book $(B, f)$ such that $L\pitchfork B$ where $L$ is a given closed, compact, orientable Legendrian surface in a closed contact $5$-manifold $(M,\xi=\textrm{Ker}(\alpha))$. By Theorem \ref{thm:existence_OB}, we may assume $(B, f)$ has Stein pages. Since they intersect transversely, we have $dim(L\pitchfork B)=0$, and so they intersect along a finite number of points. (Later we will be interested in their minimal geometric intersection.) Take any orientation on the Legendrian surface $L$. If the orientations of $L$ and $B$ are consistent at a transverse intersection point, then mark the point with plus ($+$), otherwise mark the point with minus ($-$). Continue this procedure until all the intersection points have been labelled. 
\begin{remark}
	$\bullet$ Since we assume an open book structure, $M$ must be closed.\\
	$\bullet$ Since $L$ and $B$ are compact, their intersection consists of finitely many points. (Note $B$ is compact as being the binding of an open book structure.)\\
	$\bullet$ Homology intersection of $L$ and $B$ is trivial. This is because $B$ is the boundary of a page (indeed every page) of an open book, and so $B$ has zero homology class.
\end{remark}
From the above remark, the intersection of $L$ and $B$ consists of even number of points: The number of plus points is equal to the number of minus points because $L$ and $B$ have trivial homology intersection.

Consider the pages $X_0=f^{-1}(\theta), X_1=f^{-1}(\theta+\pi)$ for $\theta \in S^1$. By genericity, we may assume $L$ transversally intersects $X_0, X_1$ and their common boundary $B$.
Let $D(X)=X_{0}\cup_{\partial} X_{1}$ denote the double of the page $X$, i.e., the union of the pages $X_{0}$ and $X_{1}$ (Here $X_{0}\cong X \cong X_{1}$, and $X_{0}$, $X_{1}$ are dual pages of each other.) Clearly, $D(X)$ is a closed folded symplectic manifold and decomposed into two Stein domains $(X_{0}, d\alpha_{0})$ and $(X_{1}, d\alpha_{1})$ where $\alpha_i=\alpha|_{X_i}$. Note that these Stein pieces induce opposite orientations on the fold $B\subset D(X)$. Consider the handle decompositions as in Figure \ref{fig:31}. 

\begin{figure}[h]
	\centering	
	\includegraphics[scale=.9]{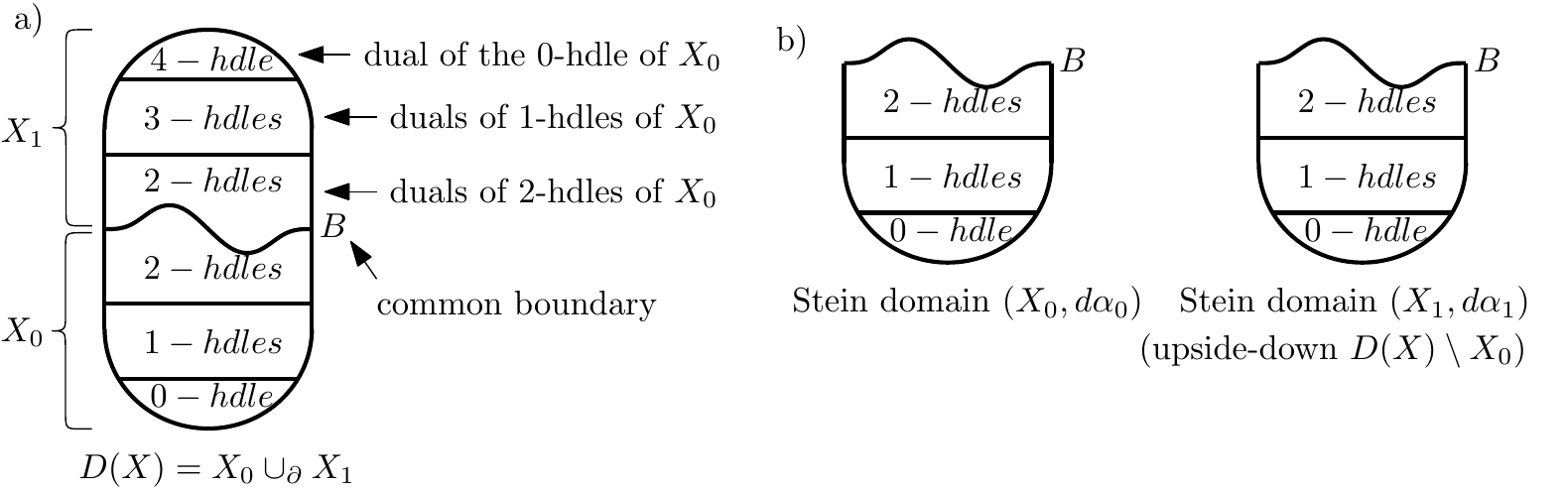}
	\caption{a) Handle decomposition of the double $D(X)$, union of the pages $X_{0}$ and $X_{1}$.  b) The Stein domains  $(X_{0}, d\alpha_{0})$ and $(X_{1}, d\alpha_{1})$. }
	\label{fig:31}
\end{figure}

Note that $dim(L)=2$ and $dim(X)=4$, so $dim(L\pitchfork D(X))=1$. That is, $L$ and $D(X)$ intersect along embedded curves due to the generic choice of $D(X)$. Let's first assume, for simplicity, that $L\pitchfork D(X)$ consists of a single  curve $K$. Also let $L\cap X_{0}=k_{0}$, $L\cap X_{1}=k_{1}$. That is, $K=k_{0}\cup k_{1}$. On $k_{0}$, take the orientation from $(+)$-points to the $(-)$-points in $L \cap B$. Similarly, on $k_{1}$, take the orientation from $(-)$-points to the $(+)$-points. See Figure \ref{fig:32}.

\begin{figure}[h]
	\centering
		\includegraphics[scale=.8]{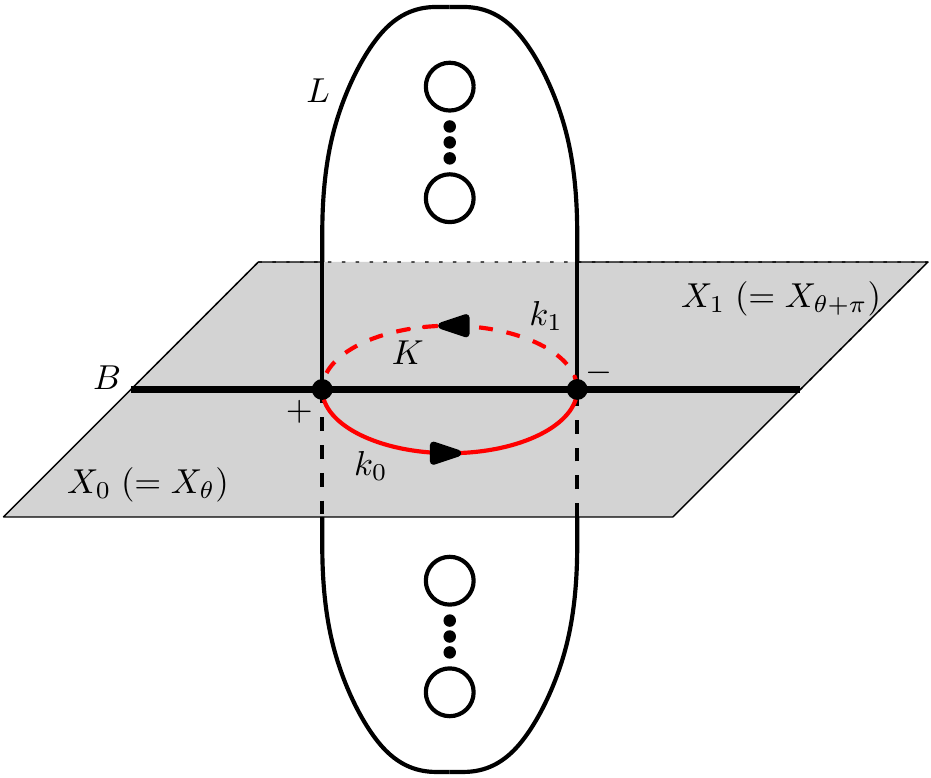}
	\caption{ Embedded Legendrian surface $L$ intersecting transversally the binding $B$ and a pair of pages $X_{0}$ and $X_{1}$. A typical knot component $K=k_{0}\cup k_{1}$ of the link of intersection of $L$ with the double $D(X)=X_{0}\cup_{\partial} X_{1}$. }
	\label{fig:32}
\end{figure}

Sketch the Legendrian arcs for $k_{0}$ and $k_{1}$ in the Stein diagrams of $X_0$ and $X_1$, respectively, and calculate Thurston-Bennequin numbers of these arcs. Summing these two numbers will give us an integer, which we'll denote by $\widetilde{tb}(K)$. In other words, we define
\begin{center}
	$\widetilde{tb}(K):= tb(k_{0})+tb(k_{1})$.
\end{center}

In the general case, the intersection of $L$ and $D(X)$ may consist of finite number of closed curves (embedded knots in $D(X)$), say $K^{1}$, $K^{2}$, ..., $K^{r}$. (Note that $K^i$'s are disjoint by transversality theorem, and so their union is a link in $D(X)$.) That is, we have
\begin{center}
	$L\pitchfork D(X)=\displaystyle{\bigsqcup_{i=1}^{r}K^{i}}$.
\end{center}
Again one can sketch the Legendrian arcs constructing the knot components of the link of the intersection of $L$ with the double $D(X)$ in the Stein diagrams of $X_0$ and $X_1$, and therefore, we obtain a diagram in Figure \ref{fig:33} describing the transverse intersection $L\pitchfork D(X)$.

\begin{figure}[h!]
	\centering
	\includegraphics[scale=.8]{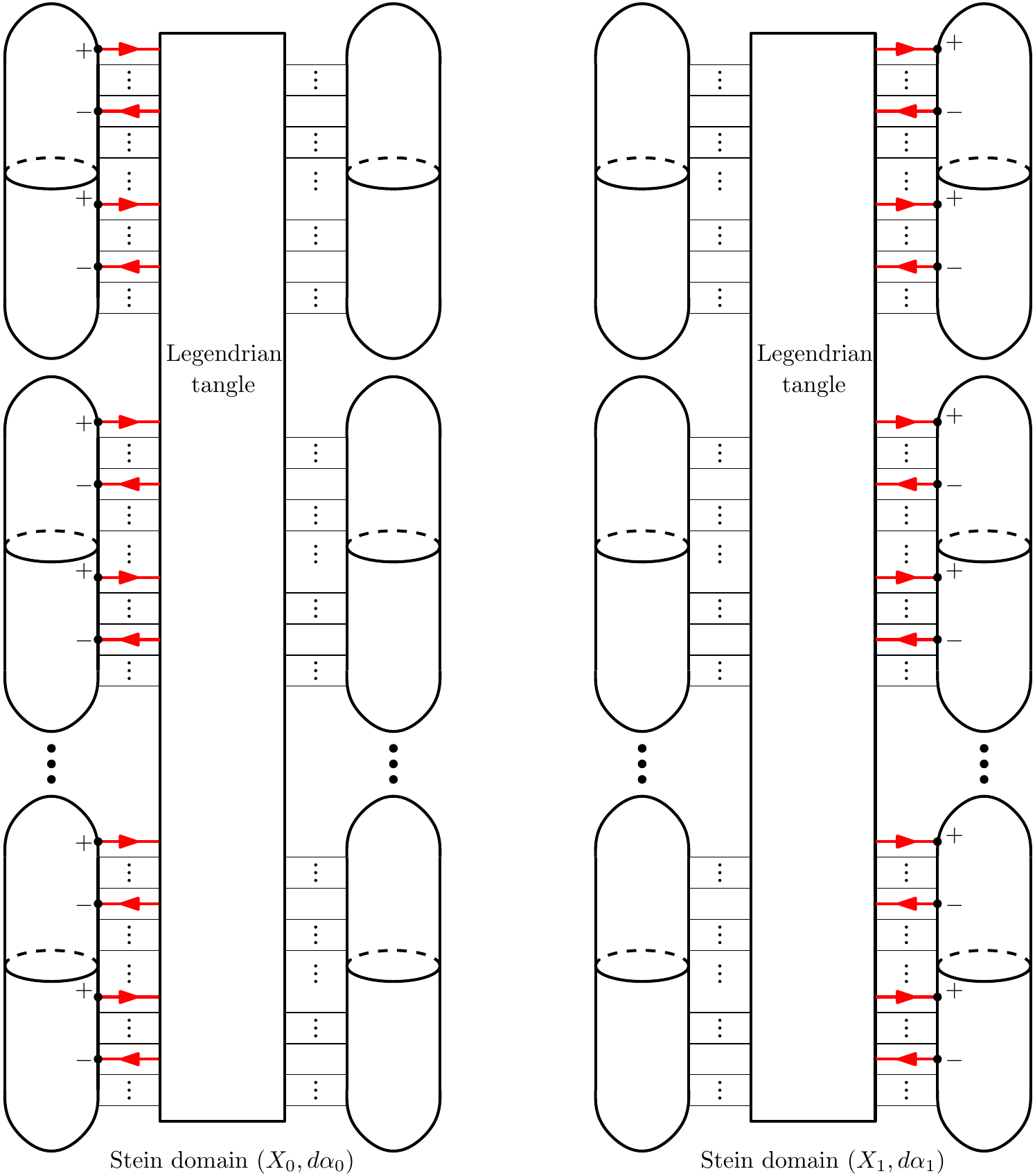}
	\caption{ Legendrian arcs (in red) constructing the knot components $K^{1}$, $K^{2}$, ..., $K^{r}$ of the link of intersection of $L$ with the double $D(X)=X_{0}\cup_{\partial} X_{1}$. }
	\label{fig:33}
\end{figure}

For all knot components $K^{i}=k_{i}^0\cup k_{i}^1$, we calculate $\widetilde{tb}(K^{i})$ as above. Summing all these together and taking the maximum of such sums by changing $L$ in its Legendrian isotopy class, one can define a number. First we need some preliminary definitions:
\begin{definition}	 \label{def:Page_crossing}
	Let $L\hookrightarrow (M^5, \xi)$ be a closed orientable Legendrian surface. Fix an admissable open book $(B, f)$ for $L$. Consider
	\begin{center}
				$[L]=\{L'\subset (M, \xi) \mid L'$ is Legendrian isotopic to $L\}$
	\end{center}
This class is called the \textbf{\textit{Legendrian isotopy class}} of $L$. Fix a page $X$ of the open book $(B, f)$, and $L'$ which is Legendrian isotopic to $L$ and transversely intersecting the double $D(X)$. Then the \textbf{\textit{page crossing number of} $L'$ \textit{with respect to} $X$} is defined as
\begin{center}
	$\mathcal{P}_{X}(L'):=\displaystyle \sum_{i=1}^{r}\widetilde{tb}(K_{i})$.
\end{center}
Lastly, we say that the double $D(X)$ \textbf{\textit{essentially intersects}} $L$ if we have $$L' \cap D(X)\neq \emptyset, \quad \forall L' \in [L].$$
\end{definition}

 We are ready to define our first invariant:
 
\begin{definition}
		Let $L\hookrightarrow (M^5, \xi)$ be a closed orientable Legendrian surface. Fix an admissable open book $(B, f)$ for $L$ and a page $X$ of $(B, f)$ such that $D(X)$ essentially intersects $L$. Then
	\begin{center}
		$M\mathcal{P}_{X}(L):=Max \left\{\mathcal{P}_{X}(L') \mid L' \in [L] \;\textrm{ and }\; L' \pitchfork D(X) \right\}$
	\end{center}
is called the \textbf{\textit{relative maximal page crossing number of $L$ with respect to $X$}}.
\end{definition}

Well-definedness of  $M\mathcal{P}_{X}(L)$ will be discussed in Section \ref{sec:well-definedness_isotopies}. Until then, $M\mathcal{P}_{X}(L)$ will be assumed to be well-defined. The following facts indicate that the most practicle way of computing $M\mathcal{P}_{X}(L)$ is working in the case of geometrically minimal intersection.

\begin{lemma} \label{lem:arc_gamma}
	Let $K=k_0 \cup k_1$ be a component of the link of transverse intersection of $L$ with the double $D(X)=X_{0}\cup_{\partial} X_{1}$  constructed using the minimal geometric intersection points of $L$ and $B$. Suppose $\gamma$ is an arc on the attaching sphere $S$ of the $1$-handle of $X_{i}$ connecting the boundary points $\partial k_i$. Then the circle $k_i \cup \gamma$ can not be a homotopically trivial in $X_i$ for each $i=0,1$.
\end{lemma}

\begin{figure}[h!]
	\centering	\includegraphics[scale=.95]{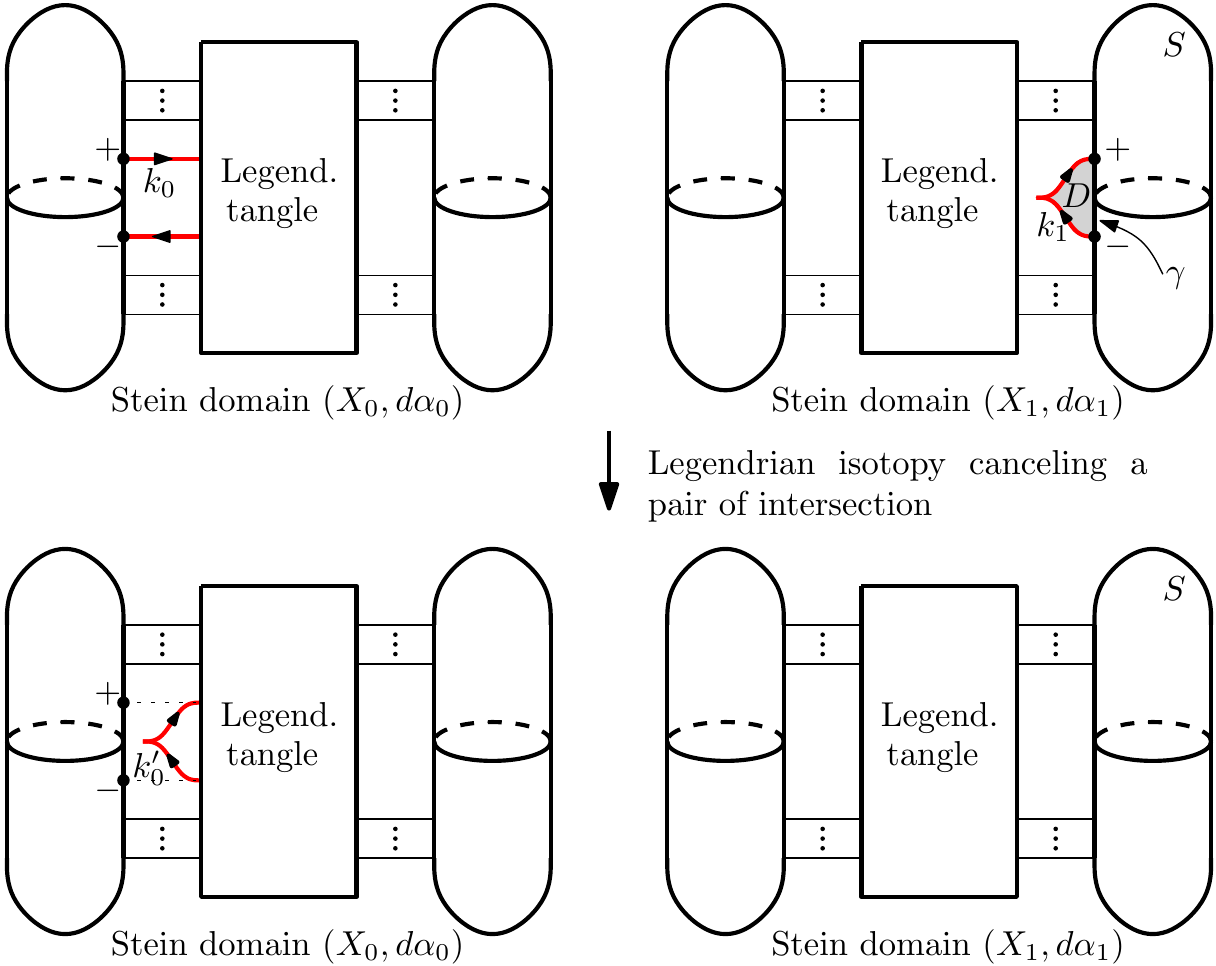}
	\vspace{0.5cm}
	\caption{ Realizing a Legendrian isotopy of $L$ (cancelling a pair of intersection points in $L\cap B$) by isotoping $L$ (through the disk $D \subset X_{1}$ enclosed by $k_{1}$ and the path $\gamma$ on the attaching sphere $S$ of the $1$-handle of $X_{1}$ joining the points ``$+$" and ``$-$") in the Stein diagrams $(X_{0}, d\alpha_{0})$ and $(X_{1}, d\alpha_{1})$. }
	\label{fig:34}
\end{figure}

\begin{proof}
	Take $i=1$ (the case $i=0$ is similar). The statement of the lemma is equivalent to say that $k_i \cup \gamma$ can not bound a disc in $X_1$. Suppose there exists such a disk $D \subset X_i$. Then using the flow of a suitable contact vector field (compactly supported in a neighborhood of $D$ in $M$ which is indeed some Darboux ball $\mathbb{D}^5$), we can Legendrian isotope $L$ until the arc $k_1$ disappears (i.e., the whole $k_1$ is transformed into $X_0$). This means that the $(\pm)$-intersection points corresponding to $\partial k_1$ is a canceling pair. Since in the new Stein pictures, there would be a less number of intersection points in $L \cap B$, this contradicts to minimality. (See Figure \ref{fig:34}.)
\end{proof}

\begin{remark}
In Lemma \ref{lem:arc_gamma}, the path $\gamma$ is chosen away from the points where other knots and arcs meet with $S$. Also in Figure \ref{fig:34}, for simplicity, $k_{1}$ is drawn with a single left cusp, but more number of cusps are also possible and threated in the same way as long as the disk $D$ exists. When we move $k_1$, this cusp (and hence the pair of intersection points ``$+$" and ``$-$") will disappear. Note that after such a canceling a pair of intersection, $\widetilde{tb}$ doesn't change, i.e.,
\begin{center}
	$tb(k_{0})+tb(k_{1})=tb(k_{0}')+tb(k_{1}')=tb(k_{0}')$ (or $=tb(k_{1}')$ in the case $i=0$).
\end{center}
\end{remark}

\begin{lemma} \label{lem:tb_negative}
		Let $K=k_0 \cup k_1$ be a component of the link of transverse intersection of $L$ with the double $D(X)$ constructed using (not necessarily minimal) geometric intersection points of $L$ and $B$. Suppose $K$ is homotopically trivial in $D(X)$. Then
			$\;\widetilde{tb}(K)= tb(k_{0})+tb(k_{1}) \leq -1$.
\end{lemma}

\begin{proof}
	By assumption there exists a disk $D \subset D(X)$ with $K=\partial D$. There are two cases: Either $k_1=\emptyset$ or $k_1\neq\emptyset$. If $k_1=\emptyset$ holds, then $K=k_0$ is a Legendrian unknot inside the Legendrian tangle in the Stein diagram of $X_0$. Therefore, it can be considered as a Legendrian unknot bounding the disk $D$ inside the Stein fillable (and so tight) boundary $\partial X_0$. But then Theorem \ref{thm:tb_bound} implies that  $\widetilde{tb}(K) \leq -1$. If $k_1\neq\emptyset$ holds, then this means that $D=D_0 \cup D_1$ where $D_0,D_1$ are disks in $X_0, X_1$, respectively, which meet along an arc $\gamma$ on the attaching spheres of the corresponding $1$-handles of $X_0$ and $X_1$. Then applying Lemma \ref{lem:arc_gamma}, one can transform $K$ to $K'$ which lies in $X_0$. Recall that $\widetilde{tb}(K)=\widetilde{tb}(K')$, that is the number $\widetilde{tb}$ does not change under the move described in the proof of Lemma \ref{lem:arc_gamma} (Figure \ref{fig:34}). Therefore, we are again in the first case above, i.e., $\widetilde{tb}(K)=\widetilde{tb}(K') \leq -1$.
\end{proof}

\begin{lemma}
	Let $k_0' \cup k_1'$ be a component of the link of transverse intersection of $L$ with the double $D(X)=X_{0}\cup_{\partial} X_{1}$  constructed using (not necessarily minimal) geometric intersection points of $L$ and $B$.
	If the knot $k_0' \cup k_1'$ is homotopically trivial in both $L$ and the double $D(X)$, then it can be ignored while computing $M\mathcal{P}_{X}(L)$. That is, $$M\mathcal{P}_{X}(L)>\mathcal{P}_{X}(L).$$	
\end{lemma}

\begin{figure}[h!]
	\centering
	\includegraphics[scale=.73]{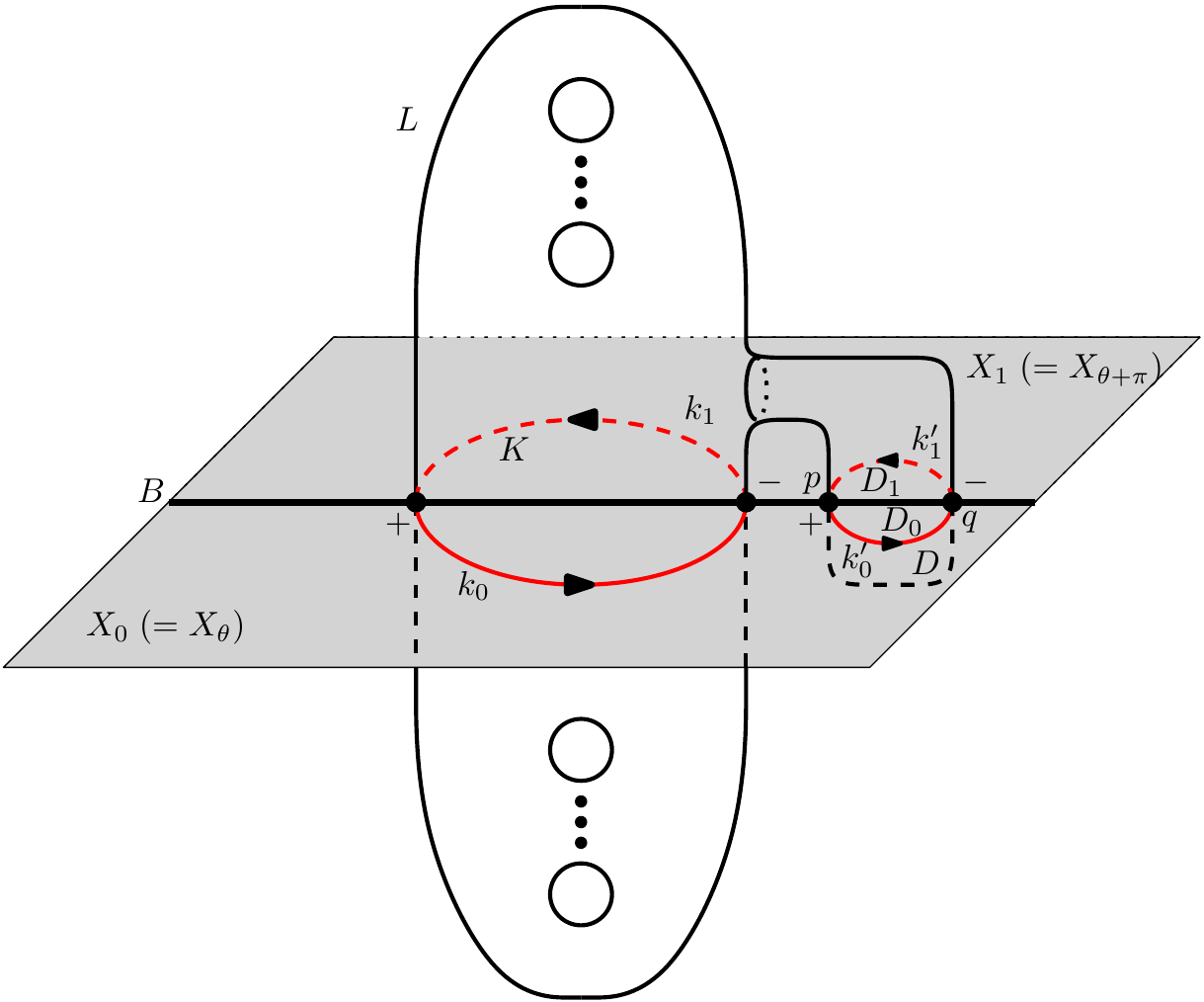}
	\vspace{0.5cm}
	\caption{ A (geometrically) cancelling pair of points ($p$ and $q$) of intersection between $L$ and the binding $B$. }
	\label{fig:35}
\end{figure}

\begin{proof}
Suppose there exists such a pair of Legendrian arcs $k_{0}'$ and $k_{1}'$ in the Stein diagrams of $X_{0}$ and $X_{1}$ whose union is homotopically trivial in both $L$ and the double $D(X)$. Therefore, there are disks $D_i \subset X_i$ such that the union $D_{0}\cup D_{1} \subset D(X)$ (enclosed by $k_{0}'\cup k_{1}'$) is not punctured by the rest of $L\cap D(X)$ and the attaching circles of the $2$-handles of $X_{0}$ and $X_{1}$, and also there is a disk $D \subset L$ bounded by $k_{0}'\cup k_{1}'$ (Figure \ref{fig:35}). Then one can get rid of the intersection arcs $k_{0}'$,  $k_{1}'$ (and so the corresponding intersection points $p$, $q$) by isotoping $B$ (and the pages of the open book) in a neighborhood $N$ of the $3$-disk enclosed by the disks $D \subset L$ and $D_{0}\cup D_{1}$ in $M$ (which is some Darboux ball $\mathbb{D}^5$) using the flow of an appropriate contact vector field compactly supported in $N\cong \mathbb{D}^5$. (See Figure \ref{fig:36}.) 

\begin{figure}[h!]
	\centering
	\includegraphics[scale=.73]{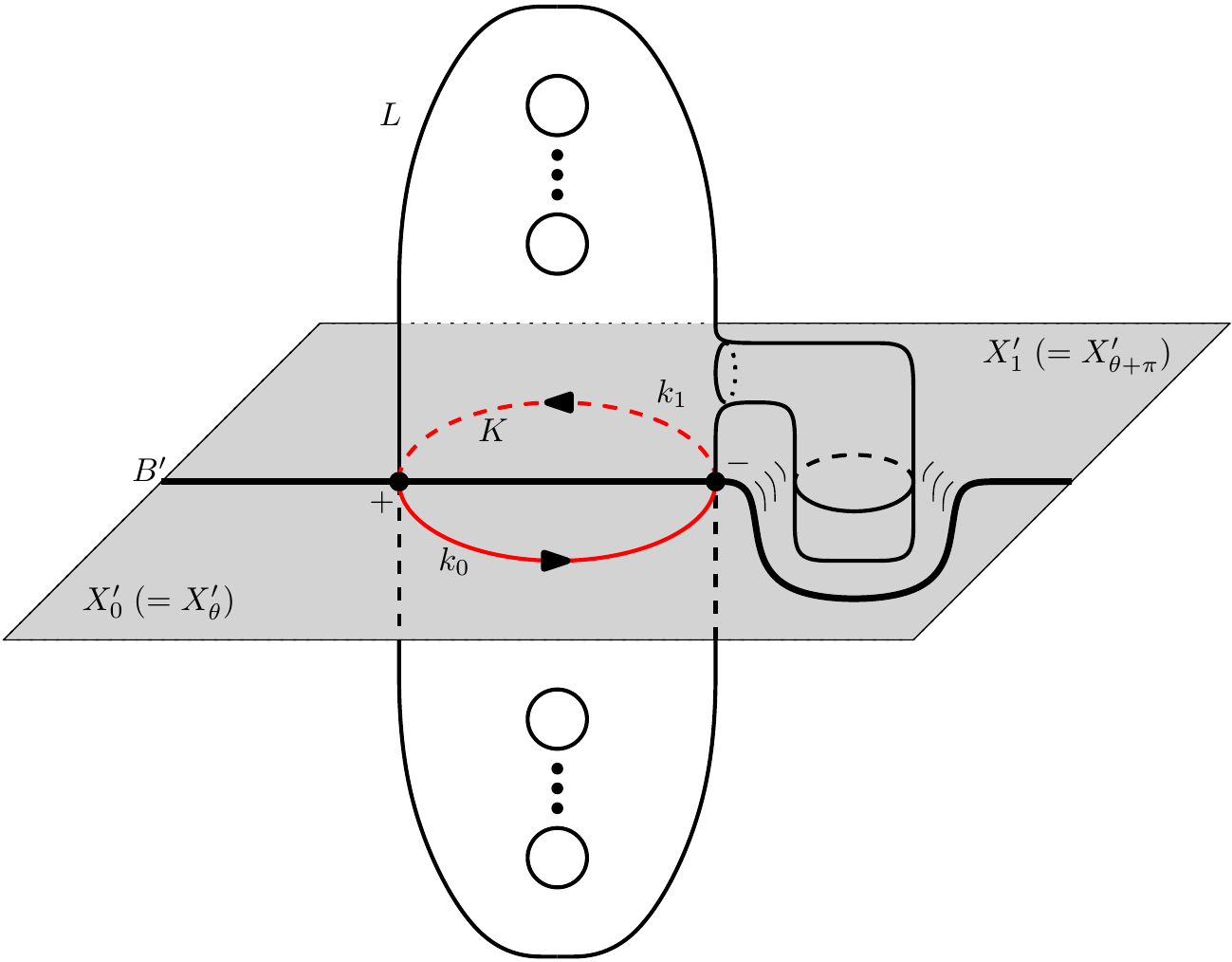}
	\caption{ Isotoping the binding $B$ and correspondingly all the pages of the open book using the flow of a compactly supported contact vector field. }
	\label{fig:36}
\end{figure}

By Lemma \ref{lem:spider} and the genericity, one can think of this isotopy results in a new open book (with the same monodromy) so that $X_i$ is transformed to a new page $X_i'$, and $B$ is transformed to a new binding $B'$. Note that this contact isotopy eliminates $k_{0}'$,  $k_{1}'$. Now we rewind this isotopy to move all the points inside the Darboux ball $\mathbb{D}^5 (\cong N)$ back to the their original positions (at the initial time). While this transform $X_0', X_1'$ and $B'$ back to their original positions, the part of $L$ in  $\mathbb{D}^5$ will be pushed further, and we get a Legendrian isotopic copy $L'$ of $L$ which does not intersect $B$ along $k_{0}'$,  $k_{1}'$. Since the isotopy is compactly supported near $k_{0}'$,  $k_{1}'$, the arcs describing $L'$ in the Stein digrams of $X_{0}$ and $X_{1}$ coinsides with the ones describing $L$ outside the Darboux ball $\mathbb{D}^5$. Therefore, to picture $L'$ in these diagrams, we simply erase the arcs $k_{0}'$,  $k_{1}'$ from the diagrams, and hence ignore their contributions to $\widetilde{tb}$. That is, we have 
	$$\widetilde{tb}(L')=\widetilde{tb}(L)-[tb(k_{0}')+tb(k_{1}')].$$
On the other hand, by Lemma \ref{lem:tb_negative}, we have $$tb(k_{0}')+tb(k_{1}')<0,$$ and so, combining this with the above equality we get $$\mathcal{P}_{X}(L')>\mathcal{P}_{X}(L).$$ Hence, $\mathcal{P}_{X}(L)$ can not be maximum, and so it is strictly less than $M\mathcal{P}_{X}(L)$.

\end{proof}


\section{Proof of Theorem \ref{thm:relative_page_crossing}} \label{sec:well-definedness_isotopies}

In this section, we will show that the number $M\mathcal{P}_{X}(L)$ is preserved under Legendrian isotopies, and also explain why it is well-defined. First, assuming it is well-defined, one can easily observe the following:

\begin{lemma} \label{lem:invariance_under_Legendrian_isotopy}
The number $M\mathcal{P}_{X}(L)$ is invariant under Legendrian isotopies of $L$.
\end{lemma}

\begin{proof}
Consider any Legendrian isotopy $L_t$ ($t\in [0,1]$) between $L=L_0$ and $L_1$. Let $X$ be a fixed page of an admissable open book $(B,f)$ for $L$ such that $D(X)$ essentially intersects $L$. Suppose that $L' \in [L]$ is a representative maximizing $\mathcal{P}_{X}$, that is,  $$M\mathcal{P}_{X}(L)=\mathcal{P}_{X}(L').$$ Since $L_1$ is Legendrian isotopic to $L$, we have $[L_1]=[L]$, that is, their Legendrian isotopy classes are the same. Therefore, $L'$ is maximizing $\mathcal{P}_{X}$ among all representatives in $[L_1]$ as well, that is,  $$M\mathcal{P}_{X}(L_1)=\mathcal{P}_{X}(L').$$ Hence, $M\mathcal{P}_{X}(L_1)=\mathcal{P}_{X}(L')=M\mathcal{P}_{X}(L)$ as required.

\end{proof}

In order to show that $M\mathcal{P}_{X}(L)$ is well-defined, first of all, one needs to understand how $\mathcal{P}_{X}(L)$ changes under possible types of Legendrian isotopies of $L$. For a fixed page $X$, there are two types of Legendrian isotopies of a given Legendrian surface $L$ which are called a \textit{regular isotopy} and an \textit{irregular isotopy}.


\subsection{Regular Isotopy}
Let $L\hookrightarrow (M^5, \xi)$ be a closed orientable Legendrian surface. Take an admissable open book $(B, f)$ for $L$. Fix a page $X$ of the open book $(B, f)$ such that $L$ is transversely intersecting the double $D(X)$. (By genericity, this is possible.) we define:

\begin{definition}
	A \textbf{\textit{regular isotopy of}} $L$ \textbf{\textit{with respect to}} $D(X)$ is a Legendrian isotopy $L_t$ ($t\in [0,1]$) of $L=L_0$ such that $L_t$ transversely intersects $D(X)$ for all $t\in [0,1]$.
\end{definition}

Under the assumptions introduced above, we have

\begin{proposition} \label{prop:regular_isotopy}
	The number $\mathcal{P}_{X}(L)$ is invariant under regular Legendrian isotopies of $L$ with respect to $D(X)$.
\end{proposition}

\begin{proof} Consider a regular Legendrian isotopy $L_t$ ($t\in [0,1]$) of $L=L_0$. By definition $L_t$ transversely intersects $D(X)$ for all $t\in [0,1]$. We need to show that $\mathcal{P}_{X}(L')=\mathcal{P}_{X}(L)$ where $L'=L_1$ is the Legendrian copy of $L$ at time $t=1$.

Let $K=k_{0}\cup k_{1}$ be any knot component in $L\pitchfork D(X)$. Since $L_t$ transversely intersects $D(X)$ for all $t\in [0,1]$, during the isotopy, $K$ is transformed through knots $K_t \in L_{t}\pitchfork D(X)$ to a knot component $K'=k'_{0}\cup k'_{1}\in L'\pitchfork D(X)$ as depicted in Figure \ref{fig:41}. (Here we think $K=K_0$, $K'=K_1$.) 

\begin{figure}[h]
	\centering
	\includegraphics[scale=.9]{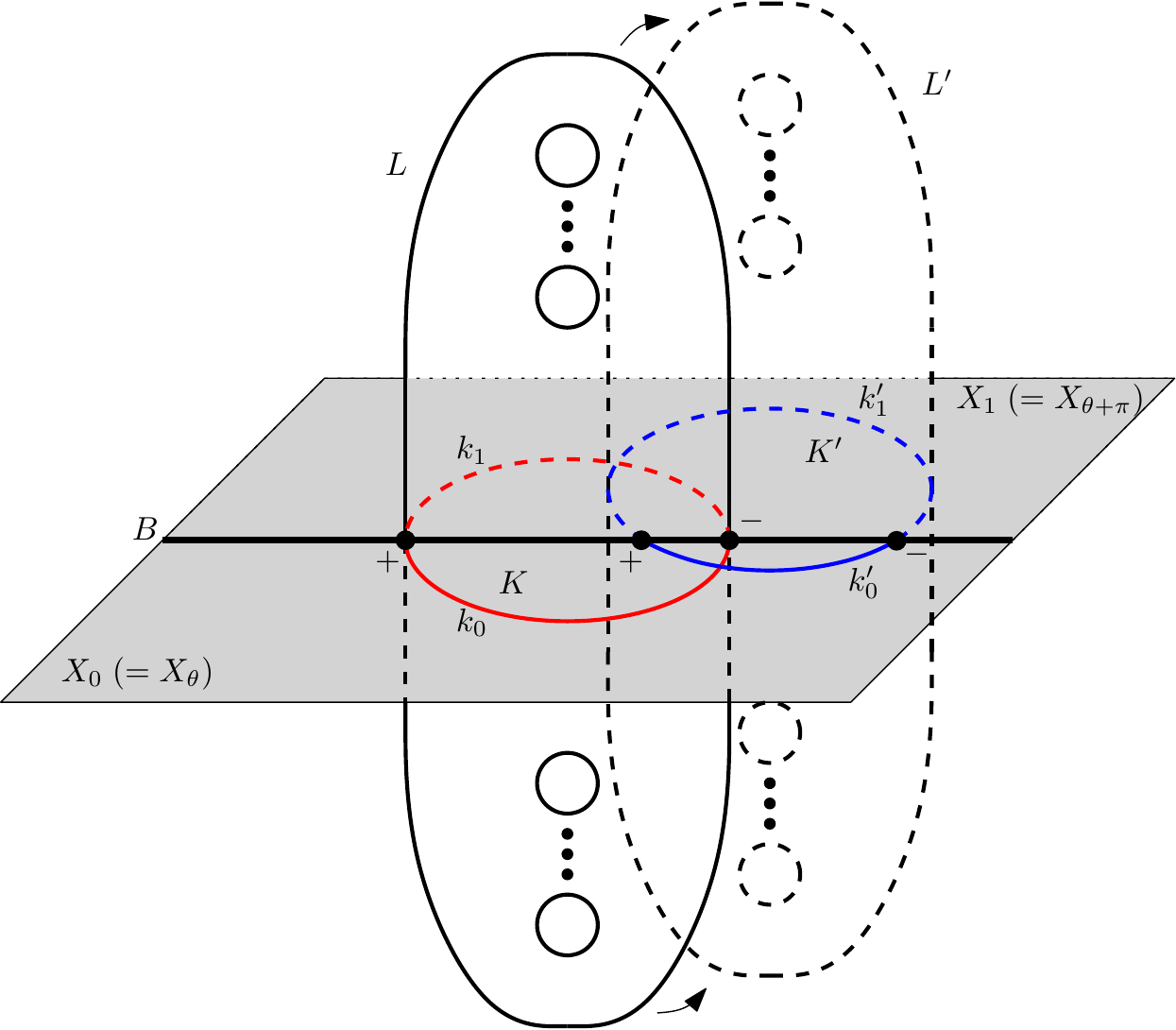}
	\caption{ A regular Legendrian isotopy $L_{t}$ taking $L_{0}=L$ to another Legendrian $L_{1}=L'$ which is still intersecting the double $D(X)=X_{0}\cup_{\partial} X_{1}$ transversely, but the new points of intersection in $L'\cap B$ are possibly different than the older ones.}
	\label{fig:41}
\end{figure}

Observe that $K_t, t\in [0,1]$ indeed defines a Legendrian isotopy from $K$ to $K'$ when we consider their arcs to be embedded Legendrian arcs inside Stein diagrams of $X_0$ and $X_1$. (See Figure \ref{fig:43} for a sample picture.)    

\begin{figure}[h]
	\centering
	\includegraphics{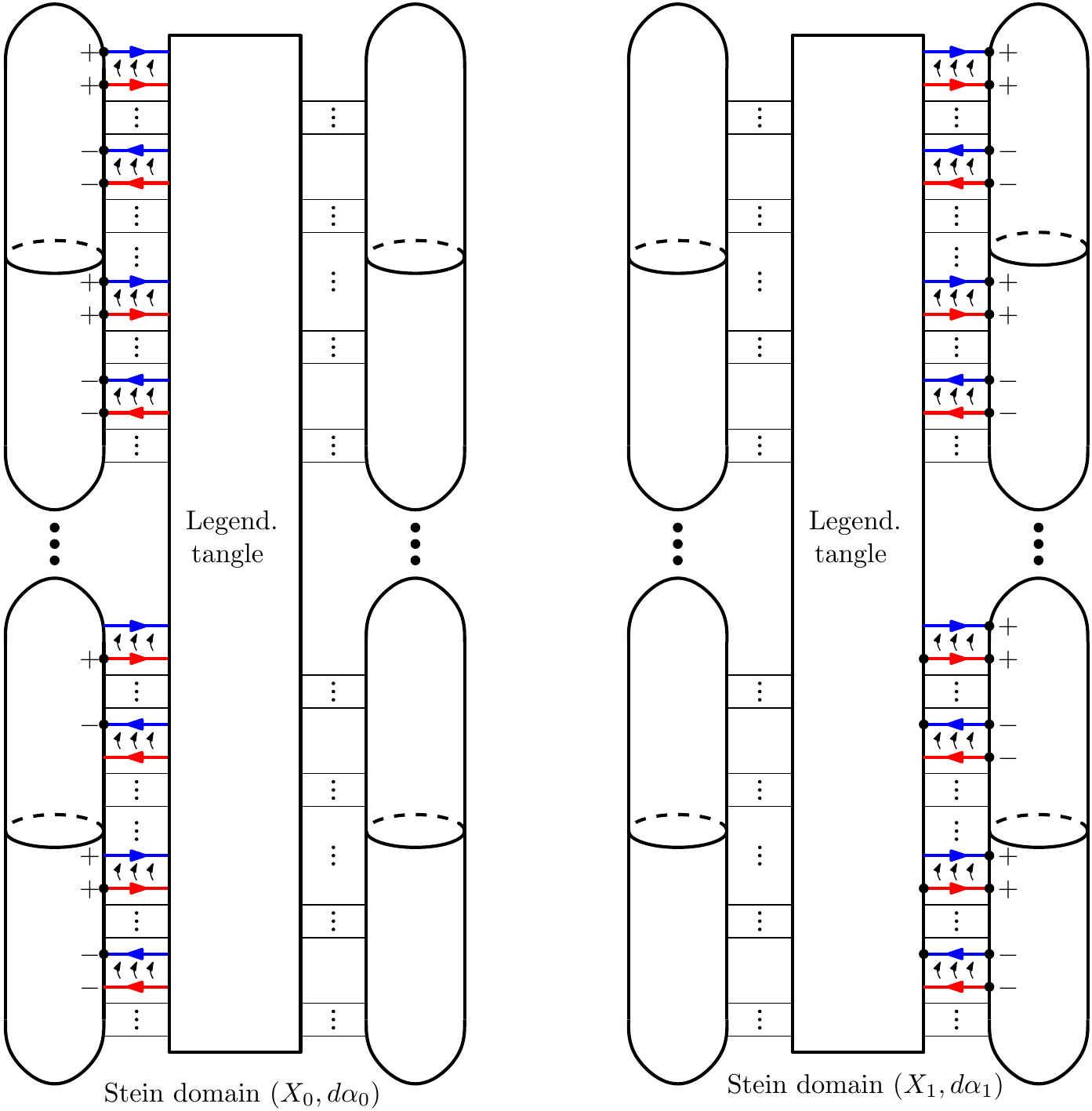}
	\caption{ Realizing a regular Legendrian isotopy $L_{t}$ ($t\in [0, 1]$) taking $L_{0}=L$ to another Legendrian $L_{1}=L'$ in the Stein diagrams of $X_{0}$ and $X_{1}$. The Legendrian arcs (in red) describing $L\cap D(X)$ are Legendrian isotopic to those (in blue) describing $L'\cap D(X)$ through a Legendrian isotopy $K_t=L_{t}\cap D(X)$, $t \in [0, 1]$. }
	\label{fig:43}
\end{figure} 

Therefore, since their union are isotopic via Legendrian ``moves'', the arcs constructing $K$ and $K'$ satisfies
\begin{center}
	$tb(k_{0})+tb(k_{1})=tb(k_{0}')+tb(k_{1}')$,
\end{center}
and so, $\widetilde{tb}(K)=\widetilde{tb}(K')$. This implies that $\mathcal{P}_{X}(L)=\mathcal{P}_{X}(L')$ because each summand of $\mathcal{P}_{X}(L)$ agrees with the corresponding summand of $\mathcal{P}_{X}(L')$ by the above discussion.

\end{proof}

\begin{remark}
Observe that all the arguments in the proof of Proposition \ref{prop:regular_isotopy} work whenever we take a Legendrian representative $L$ from the Legendrian isotopy class $[L]$ which transversely intersects the double $D(X)$. In particular, if $L$ (which we start with at the beginning of the proof) is itself maximazing all such possible page crossing numbers, i.e., if $$M\mathcal{P}_{X}(L)=\mathcal{P}_{X}(L),$$ then the same will be also true for $L'$. As a result, we have  $M\mathcal{P}_{X}(L)=M\mathcal{P}_{X}(L')$. Hence, this reproves Lemma \ref{lem:invariance_under_Legendrian_isotopy} in the case of regular Legendrian isotopies with respect to $D(X)$.
\end{remark}


\subsection{Irregular Isotopy}
Once again let $L\hookrightarrow (M^5, \xi)$ be a closed orientable Legendrian surface. Take an admissable open book $(B, f)$ for $L$. Fix a page $X$ of the open book $(B, f)$ such that $L$ is transversely intersecting the double $D(X)$.  We define:

\begin{definition}
	An \textbf{\textit{irregular isotopy of}} $L$ \textbf{\textit{with respect to}} $D(X)$ is a Legendrian isotopy $L_t$ ($t\in [0,1]$) of $L=L_0$ such that $L'=L_1$ still transversely intersects $D(X)$ but the new intersection set $L'\cap D(X)$ is obtained from $L\cap D(X)$ via a sequence of births or deaths of intersection knots or due to degenerations of knots in $L\cap D(X)$.
\end{definition}

\begin{proposition} \label{prop:irregular_isotopy}
	During irregular Legendrian isotopies of $L$ with respect to $D(X)$, there can not be any births or deaths of nontrivial intersection knots with $D(X)$. Moreover, under such isotopies, the number $\mathcal{P}_{X}(L)$ makes only finite jumps due to births or deaths of unknots and degenerations of knots in $L\cap D(X)$.
\end{proposition}

\begin{proof} Consider an irregular Legendrian isotopy $L_t$ ($t\in [0,1]$) of $L=L_0$. By definition, $L_t$ does not transversely intersect $D(X)$ for all $t\in [0,1]$. But generically almost all intersection will be transverse. After a small perturbation of the isotopy $L_t$ (if necessary) but still calling the resulting isotopy $L_t$, one may assume that there are numbers $0<t_0<t_1<\cdots <t_r<1$ so that except finitely many $L_{t_i},  (i=0,1,...,r)$, any other $L_t$ intersects $D(X)$ transversely. Therefore, for the second statement, one needs to show that there exists $N\in \mathbb{N}$ such that $$|\mathcal{P}_{X}(L')-\mathcal{P}_{X}(L)|<N$$ where $L'=L_1$ is the Legendrian copy of $L$ at time $t=1$.

Let us consider the case when we pass from time $t=0$ to $t=t_0+\epsilon$ for $\epsilon<t_1-t_0$. (the discussion for passing $t=t_i-\epsilon$ to $t=t_i+\epsilon$ is similar.) First of all, comparing to those in $L\cap D(X)$ if there are new unknots (births) in $L_{t_0+\epsilon}\cap D(X)$ (they necesarrily bound disks in $D(X)$ by admissibility assumption), then these births arise as an unknot $K$ which may (or may not) bound a disk $D'$ in $L_{t_0+\epsilon}$, but they must bound a disk $D$ in $D(X)$ as depicted in Figure \ref{fig:figure_birth_unknot}. The existence of the disk $D$ and  Lemma \ref{lem:tb_negative} implies that $\widetilde{tb}$ of all these unknots are negative, and so whenever such an unknot arises, this will decrease the number  $\mathcal{P}_{X}$. Similarly, comparing to those in $L\cap D(X)$ if there are missing unknots (deaths) in $L_{t_0+\epsilon}\cap D(X)$ (which were bounding disks in $D(X)$), then these will increase the number  $\mathcal{P}_{X}$. 

\begin{figure}[h]
	\centering
	\includegraphics[scale=.9]{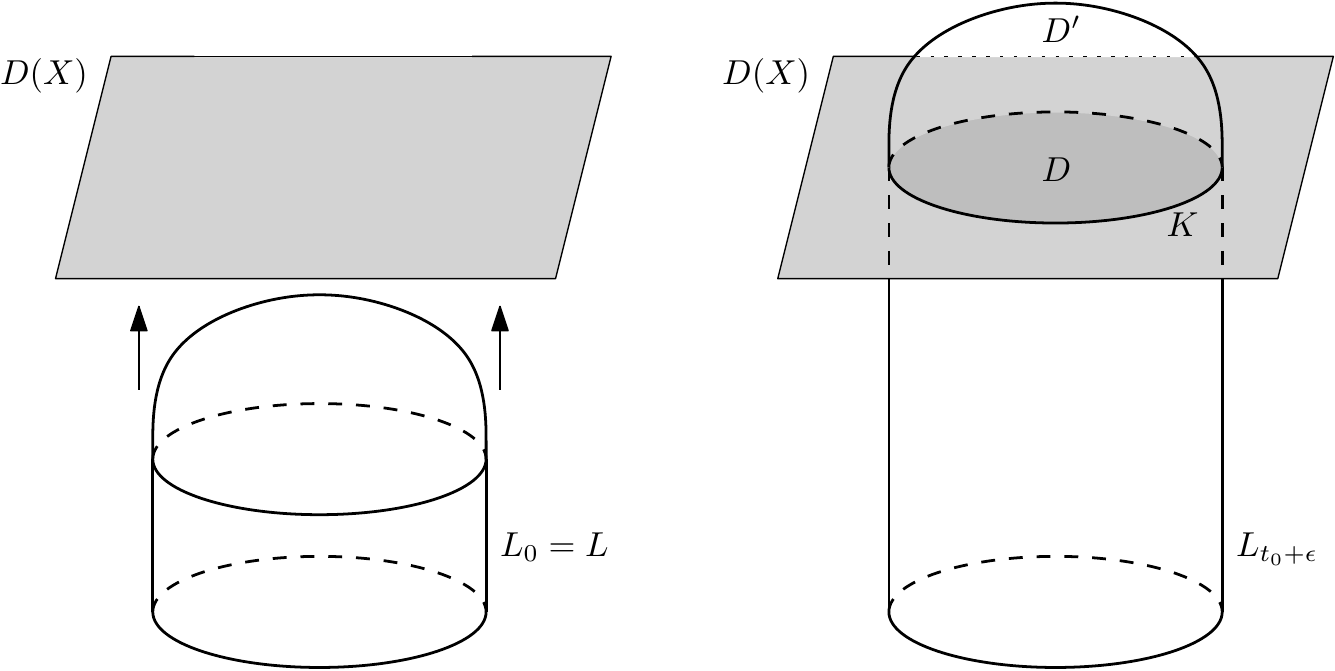}
	\caption{A birth of an unknot $K$ during an irregular Legendrian isotopy $L_{t}$ ($t\in [0, 1]$). $K$ bounds a disk $D'$ in $L_{t_0+\epsilon}$ and a disk $D$ in $D(X)$.}
	\label{fig:figure_birth_unknot}
\end{figure} 

Note that these births can not be non-trivial knots in $L_{t_0+\epsilon}$ and none of the new knots can link to a knot in $L_{t_0+\epsilon}\cap D(X)$ which was also in $L\cap D(X)$ because otherwise there would be a time $s\in (0,t_0+\epsilon)$ such that $L_s$ is not an embedding which is a contradiction. Similarly, none of the missing knots (deaths) in $L_{t_0+\epsilon}\cap D(X)$ can be a non-trivial knot and can link to a knot in $L\cap D(X)$ at the time $t=0$. To sum up, births and deaths in $L_{t_0+\epsilon}\cap D(X)$ can occur only along unknots, say $U_1,.., U_b$ are the births and $U'_1,.., U'_d$ are the deaths. Note the total number of these births and deaths is finite due to smoothness and compactness arguments. Therefore, when passing from $t=0$ to $t=t_0+\epsilon$, the change in $\mathcal{P}_X$ due to births and deaths can be at most $$|\widetilde{tb}(U_1)+\cdots +\widetilde{tb}(U_b)-(\widetilde{tb}(U'_1)+\cdots +\widetilde{tb}(U'_d))|.$$
Next, we will discuss the case when there are degenerations transforming some collection of knots in $L\cap D(X)$ to new ones in $L_{t_0+\epsilon}\cap D(X)$. Degenerations may arise as either unifications or separations which are exactly the opposite of each other, and so it suffices to understand one of them. A typical situtation of unification is the following: Suppose that  the intersection knots  $K_{1}, K_{2} \in L\cap D(X)$ degenerate during the isotopy and a new intersection knot $K \in L'\cap D(X)$ arises while $K_{1}, K_{2} $ disappear (unify) as depicted in Figure \ref{fig:42} and Figure \ref{fig:44}  where for simplicity we assume that there is a single degeneration and take $L'=L_{t_0+\epsilon}$. In the Stein diagrams of $X_{0}$ and $X_{1}$, this degeneration and the creation of $K$ correspond to bringing the $+, -$ points together on the attaching spheres of $1$-handles, and then taking a Legendrian connect sum of $K_{1}, K_{2}$ along an appropriate Legendrian band. (See Figure \ref{fig:44}.) We note that such a degeneration may also appear away from the binding, that is, it can occur in the Legendrian tangle of one of the Stein diagrams of either $X_0$ or $X_1$.\\

Observe that during an unification (resp. a separation), the number $\widetilde{tb}$ decreases (resp. increases) by $1$. More precisely, in Figure \ref{fig:45}, some different ways of obtaining a Legendrian connected sum of the knots $K_{1}$ and $K_{2}$ along appropriate Legendrian bands (in red) are given. Any Legendrian band connecting $K_{1}$ and $K_{2}$ may arise when $K_{1}, K_{2}$ unify (and a new intersection knot $K$ borns as $K_{1}\#K_{2}$) during an irregular Legendrian isotopy. Eqivalently, any Legendrian band can occur when $K$ separates and decomposes as the disjoint union of $K_1, K_2$.  It is not hard to show that no matter which Legendrian band is used (arises) during a creation (resp. separation) of $K=K_1\#K_2$), the number $\widetilde{tb}$ always decreases (resp. increases) by 1 because gluing with a Legendrian band always introduces one additional left cusp (see Figure \ref{fig:45}). That is, the following always holds:
	
	\begin{center}
		$\widetilde{tb}(K)=\widetilde{tb}(K_1\#K_2)=\widetilde{tb}(K_{1})+\widetilde{tb}(K_{2})-1$.
	\end{center}
	
To summarize, when passing  from time $t=0$ to $t=t_0+\epsilon$, if there are $M_u$ unifications and $M_s$ separations (note the total number of degenerations is again finite by smoothness and compactness arguments), then the change in $\mathcal{P}_X$ due to these degenerations can be at most $$|M_u-M_s|.$$ Combining with the births and deaths argument above, we conclude that the change in $\mathcal{P}_X$ (when passing  from time $t=0$ to $t=t_0+\epsilon$) is finite and satisfies
$$|\mathcal{P}_{X}(L_{t_0+\epsilon})-\mathcal{P}_{X}(L)|<N_{0}:=|\widetilde{tb}(U_1)+\cdots +\widetilde{tb}(U_b)-(\widetilde{tb}(U'_1)+\cdots +\widetilde{tb}(U'_d))|+|M_u-M_s|.$$

\begin{figure}[h!]
	\includegraphics[scale=.9]{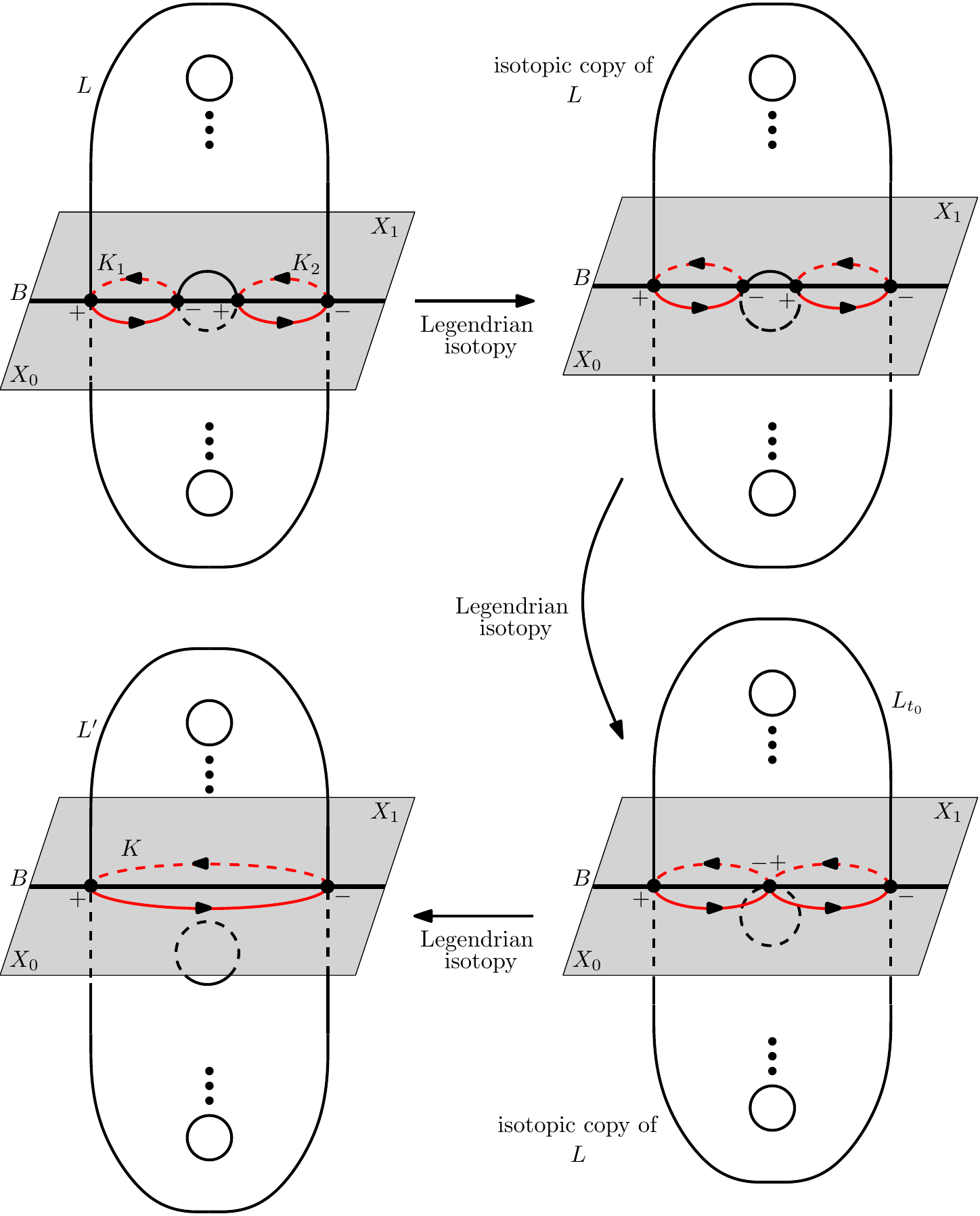}
	\caption{ A typical degeneration (unification) during an irregular Legendrian isotopy $L_{t}$ taking $L_{0}=L$ to another Legendrian $L_{1}=L'$ which also intersects the double $D(X)=X_{0}\cup_{\partial} X_{1}$ transversely, but the new arcs of intersection in $L'\cap D(X)$ are different than the older ones. The unification of $K_1$ and $K_2$ results in the creation of $K$. Note that traveling in the opposite direction (i.e., from $t=1$ to $t=0$) describes a typical separation of $K$ into $K_1$ and $K_2$.}
	\label{fig:42}
\end{figure}

As a result, repeating the above argument for each $t_i$ with $0<t_0<t_1<\cdots <t_r<1$, we conclude that during the irregular Legendrian isotopy  $L_{t}$ ($t\in [0, 1]$), the total change in $\mathcal{P}_X$ is finite. More precisely,  we have 
$$|\mathcal{P}_{X}(L')-\mathcal{P}_{X}(L)|=|\mathcal{P}_{X}(L_1)-\mathcal{P}_{X}(L_0)|<N$$ where $N:=N_{0} +\cdots + N_i+ \cdots +N_r$. Here, for each $i=1,...,r$, the bound $N_i$ is obtained (similarly to $i=0$ case above) by analyzing corresponding births/deaths and degenerations occuring when passing from $t=t_i-\epsilon$ to $t=t_i+\epsilon$. 

\end{proof}

\begin{figure}
	\centering
	\includegraphics{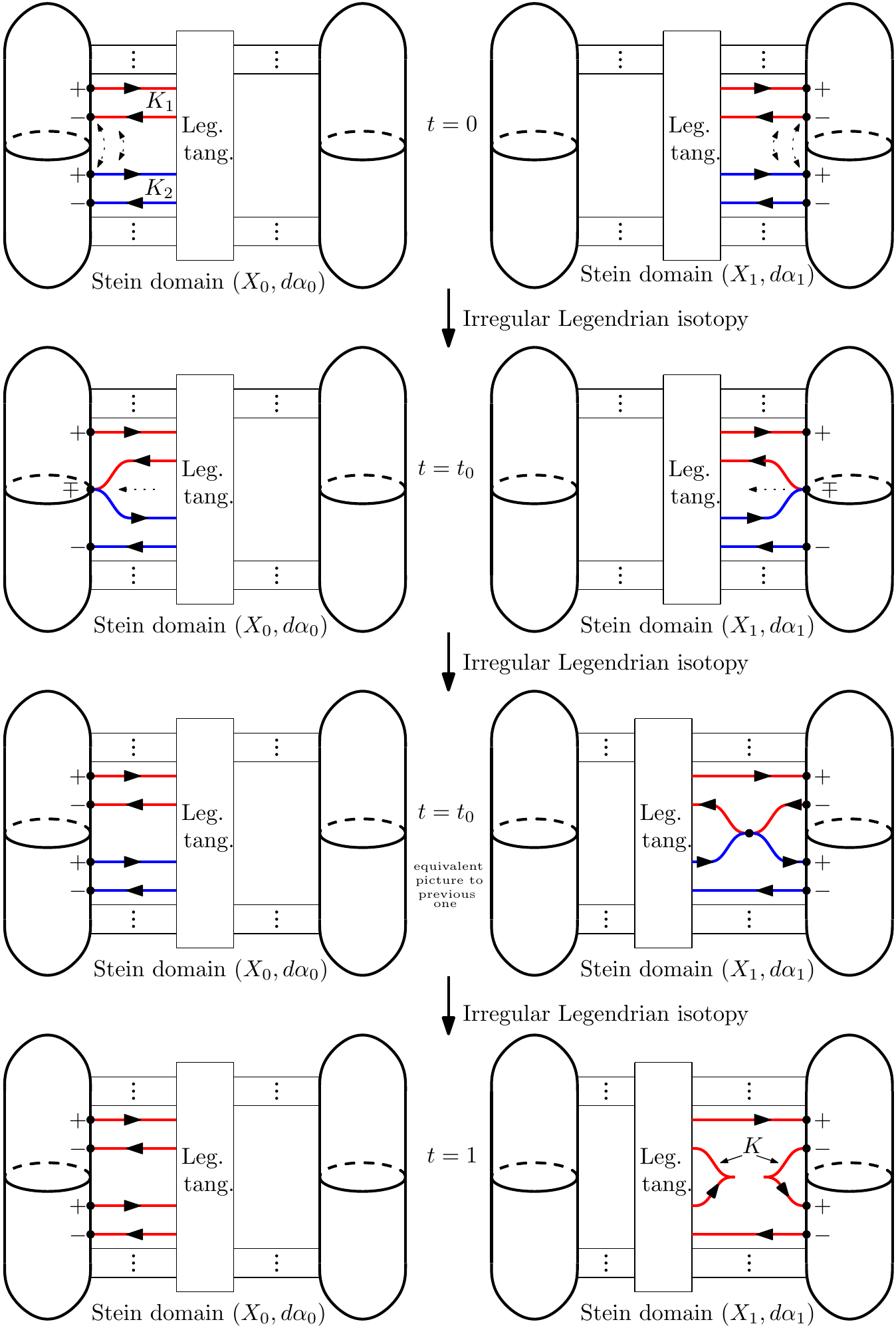}
	\caption{ Realizing a degeneration (unification) of $K_{1}$, $K_{2} \in L\cap D(X)$ and the creation of $K \in L'\cap D(X)$ during an irregular Legendrian isotopy $L_{t}$, $t \in [0, 1]$. ($L_{0}=L, L_{1}=L'$ and $L_{t_{0}}$ is not transverse to $D(X)$.) }
	\label{fig:44}
\end{figure} 

\begin{figure}[h]
	\centering
	\includegraphics[scale=.9]{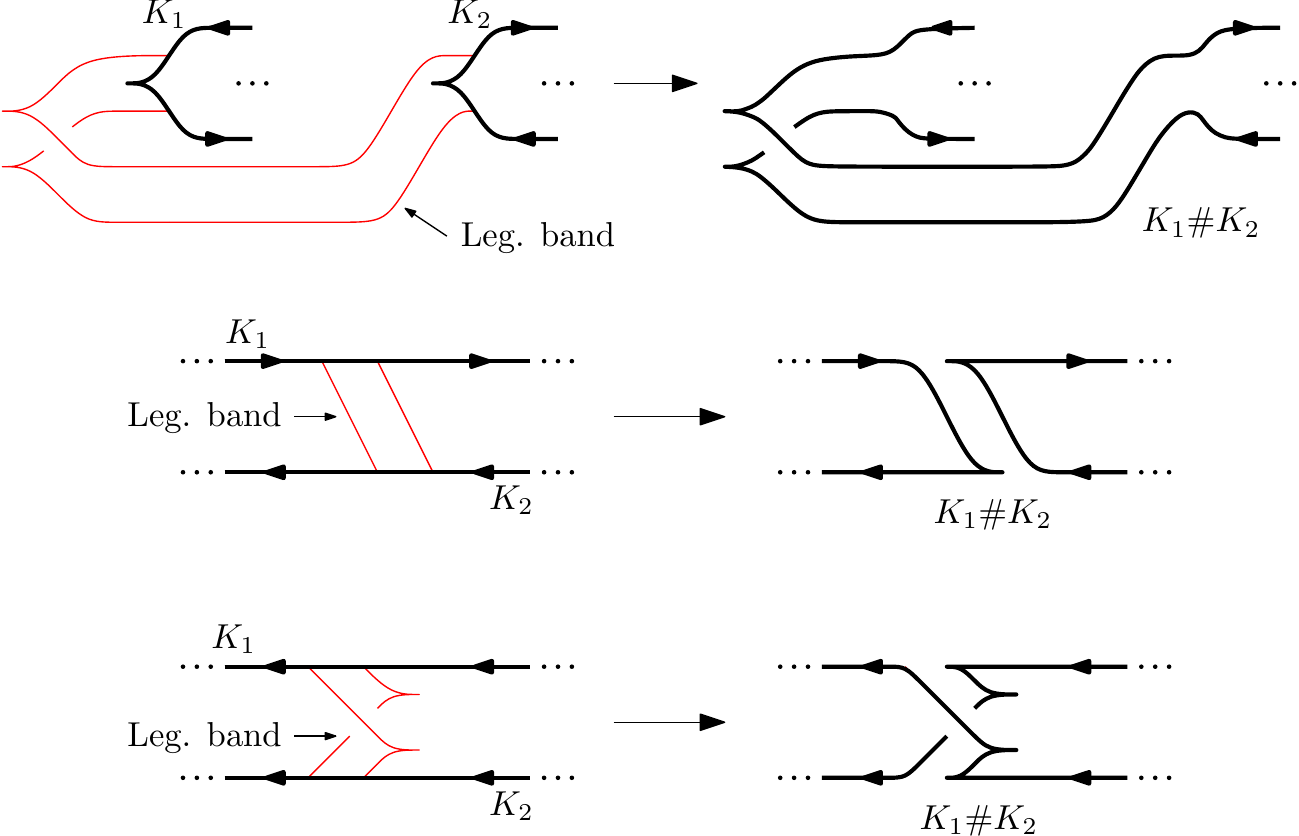}
	\caption{ Some possible ways (but not all) of obtaining a Legendrian connected sum $K_{1}\#K_{2}$ of the knots $K_{1}$ and $K_{2}$ along appropriate Legendrian bands (in red). }
		\label{fig:45}
\end{figure}

\begin{theorem} \label{thm:well-definedness_relative}
The number $M\mathcal{P}_{X}(L)$ is well-defined.
\end{theorem}

\begin{proof} Recall the definition of relative invariant:
	\begin{center}
		$M\mathcal{P}_{X}(L):=Max \left\{\mathcal{P}_{X}(L') \mid L' \in [L] \;\textrm{ and }\; L' \pitchfork D(X) \right\}$
	\end{center}
where $(B, f)$ is an admissable open book for $L$ and a page $X$ is chosen so that $D(X)$ essentially intersects $L$.	Therefore, for any $L' \in [L]$, the intersection $L'\cap D(X)$ is non-empty which implies that the set $$\mathcal{P}_{X}([L]):=\left\{\mathcal{P}_{X}(L') \mid L' \in [L] \;\textrm{ and }\; L' \pitchfork D(X) \right\} \subset \mathbb{Z}$$ is a non-empty subset.  In order to check well-definedness of $M\mathcal{P}_{X}(L)$, we need to verify that the set $\mathcal{P}_{X}([L])$ attains its maximum value. That is, one needs to show that  there exists $L_{max} \in [L]$ such that $$Max(\mathcal{P}_{X}([L]))=\mathcal{P}_{X}(L_{max})<\infty.$$
Equivalently, we need to show that the number $\mathcal{P}_{X}(L)$ can not be made arbitrarly large under Legendrian isotopies of $L$.
By Proposition \ref{prop:regular_isotopy}, $\mathcal{P}_{X}(L)$ is invariant, and so, can not be made arbitrarly large under regular Legendrian isotopies of $L$. Therefore, it suffices to consider irregular Legendrian isotopies of $L$. By Proposition \ref{prop:irregular_isotopy}, we know that the jump in $\mathcal{P}_{X}(L)$ under any irregular isotopy is finite. Consider the subset $$[L]_{min} \subset [L]$$ of all Legendrian representatives of $L$ in the class $[L]$ which intersects $D(X)$ transversely and minimally. In other words, for any $\widetilde{L} \in [L]_{min}$, the set $\widetilde{L} \pitchfork D(X)$ is a link in $D(X)$ contains no unknot components due to a birth which may arise during an irregular Legendrian isotopy. Clearly, by undoing any such isotopy (if needed) one can get rid of any such unknots (i.e., replacing any birth with the corresponding death), any Legendrian representative $L \in [L]$ intersecting $D(X)$ transversely can be transformed to some $\widetilde{L} \in [L]_{min}$. That is, there is a map
$$\Psi:[L] \longrightarrow [L]_{min}, \quad \Psi(L)=\widetilde{L}.$$
From its construction, it is clear that $\mathcal{P}_{X}(L')\leq \mathcal{P}_{X}(\Psi(L'))$ for any $L' \in [L]$ with $L' \pitchfork D(X)$. Therefore, we have $Max(\mathcal{P}_{X}([L]))=Max(\mathcal{P}_{X}([L]_{min}))$, and hence, it suffices to focus on the set $[L]_{min}$, i.e., if $L_{max}$ exists, then  $L_{max}\in [L]_{min}$. Equivalently, one needs to show that  there exists $L_{max} \in [L]_{min}$ such that $$Max(\mathcal{P}_{X}([L]_{min}))=\mathcal{P}_{X}(L_{max})<\infty.$$

Now inside $[L]_{min}$ consider the subset $[L]^o_{min}\subset [L]_{min}$ which consists of all Legendrian representative of $\widetilde{L} \in [L]_{min}$ such that there exists a knot component $K$ in the link $\widetilde{L} \cap D(X)$ (of transverse minimal intersections) which separates into two knots $K_1$ and $K_2$ (via some irregular Legendrian isotopy with respect to $D(X)$) such that at least one of $K_i$'s (say $K_2$) is a homotopically nontrivial unknot in $\widetilde{L}$ and does not link to any other components of the resulting link of intersection. Equivalently, $K_2$ does not bound a disk in $\widetilde{L}$ but it bounds a disk $D$ in $D(X)=X_0 \cup_{\partial}X_1$ which is not punctured with any other knot in the Stein diagrams of $X_0$ and $X_1$. Such a knot component $K$ will be called \textbf{\textit{decomposable}}. Given $\widetilde{L} \in [L]^o_{min}$, find all decomposable knots $K \in \widetilde{L} \cap D(X)$ and the corresponding $K_2$'s and $D$'s mentioned above. Composing irregular Legendrian isotopies separating $K$'s into $K_1$'s and $K_2$'s with suitable Legendrian isotopies compactly supported in small neighboorhoods of $D$'s, one can get rid of all these $K_2$'s, and repeating this argument (if necessary) eventually we obtain a Legendrian representative $$\bar{L} \in [L]_{min} \setminus [L]^o_{min}.$$
Recall that separation of a link component increases $\widetilde{tb}$ by $1$, and also erasing a Legendrian unknot (corresponding $K_2$) from Stein diagrams increases $\widetilde{tb}$ at least by $1$. Therefore, for any $\bar{L}$ obtained from $\widetilde{L} \in [L]^o_{min}$ as above, the following always holds:
$$\mathcal{P}_{X}(\widetilde{L}) < \mathcal{P}_{X}(\bar{L}).$$
This means that if $L_{max}$ exists, then it must be true that $L_{max} \in [L]_{min} \setminus [L]^o_{min}$. Equivalently, in order to prove the theorem, one needs to show that  there exists $L_{max} \in [L]_{min} \setminus [L]^o_{min}$ such that 
$$Max(\mathcal{P}_{X}([L]_{min} \setminus [L]^o_{min}))=\mathcal{P}_{X}(L_{max})<\infty.$$

\medskip
To proceed further, we need a partial order relation on the set of equivalence classes of links in $D(X)$ consisting of all possible intersections of $D(X)$ with elements in $[L]_{min} \setminus [L]^o_{min}$. More precisely, consider the set of links in $D(X)$ defined by

\begin{center}
		$\Lambda:=\left\{ L \pitchfork D(X) \mid L \in [L]_{min} \setminus [L]^o_{min} \right\}$.
\end{center}

As discussed in the earlier sections, every element (link) $K \in \Lambda$ can be realized as the union of collections $k_0, k_1$ of Legendrian arcs drawn in the Stein diagrams of $X_0, X_1$, respectively. We will write $\Vert K\Vert=\Vert K'\Vert$ and say that two links $K,K' \in \Lambda$ are \textbf{\textit{isotopy equivalent}} if each corresponding collections $k_i, k'_i$ ($i=0,1$) are related via Legendrian Reidemeister moves and their modifications (the ones which does not change $tb$) for Stein digrams described in \cite{G8}.

\begin{remark}
Note that from the definition of page crossing number, for any  $L \in [L]_{min} \setminus [L]^o_{min}$, we have $$\mathcal{P}_{X}(L)=\widetilde{tb}(L \pitchfork D(X))=\widetilde{tb}(K).$$ Therefore, showing the existence of an $L_{max} \in [L]_{min} \setminus [L]^o_{min}$ maximizing $\mathcal{P}_X$ is equivalent to showing the existence of a $K_{max} \in \Lambda$ maximizing $\widetilde{tb}$.
\end{remark}

Next we will define a partial order relation on the set 

\begin{center}
		$A:=\left\{ \;\; \Vert K \Vert \;\; \mid  \;\; K \in \Lambda \;\;\right\}$.
\end{center}

\begin{definition} \label{def:partial_order}
Let $K_0,K_1 \in \Lambda$, so there exist $L_0,L_1 \in  [L]_{min} \setminus [L]^o_{min}$ so that $K_i=L_i \pitchfork D(X)$. We will write $\Vert K_0 \Vert \preceq \Vert K_1 \Vert$ if
\begin{itemize}
\item[(I)] There is a regular or an irregular Legendrian isotopy $L_t$ $(t\in[0,1])$ with respect to $D(X)$ having only separating degenerations such that  whenever $L_t$ is transverse to $D(X)$, we have $$L_t \in [L]_{min}\setminus [L]^o_{min}.$$
\item[(II)] $\mathcal{P}_{X}(L_0) \leq \mathcal{P}_{X}(L_1)$ (or equivalently, $\widetilde{tb}(K_0)\leq \widetilde{tb}(K_1)$.)
\end{itemize}
\end{definition}

\begin{lemma} 
The pair $(A, \preceq)$ is a partially ordered set.
\end{lemma}

\begin{proof} 
\textit{Reflexivity}: For a given $\Vert K \Vert \in A$, consider any representative $K \in \Vert K \Vert$ and corresponding $L\in [L]_{min} \setminus [L]^o_{min}$, i.e., $K=L \pitchfork D(X)$. Then one can consider the trivial Legendrian isotopy fixing all the points on $L$ for all time $t$. The second condition is also clear. Therefore, $$\Vert K \Vert \preceq \Vert K \Vert.$$

\noindent \textit{Anti-symmetry}: Suppose $\Vert K_0 \Vert \preceq \Vert K_1 \Vert$ and $\Vert K_1 \Vert \preceq \Vert K_0 \Vert$ for $\Vert K_0 \Vert, \Vert K_1 \Vert \in A$. Immediately, we observe $\widetilde{tb}(K_0)\leq \widetilde{tb}(K_1)$ and $\widetilde{tb}(K_1)\leq \widetilde{tb}(K_0)$, and so $$\widetilde{tb}(K_0)=\widetilde{tb}(K_1).$$
Consider any representatives $K_0 \in \Vert K_0 \Vert, K_1\in  \Vert K_1 \Vert$ and the corresponding 
$L_0, L_1 \in [L]_{min}\setminus [L]^o_{min}$ which are connected via a Legendrian isotopy $L_t$ such that  whenever $L_t$ is transverse to $D(X)$, we have $L_t \in [L]_{min}\setminus [L]^o_{min}$.\\

If $L_t$ is a regular Legendrian isotopy with respect to $D(X)$, then $L_t$ transversely intersects $D(X)$ for all $t$. In particular, this implies that $L_t \cap D(X)$ is minimal and has no decomposable components for all $t$ because $L_0 \in [L]_{min}\setminus [L]^o_{min}$. Also observe $L_t$ induces an isotopy $K_t:=L_t \pitchfork D(X)$ (between $K_0$ and $K_1$) whose respective restrictions $K_t \cap X_i$ ($i=0,1$) defines Legendrian isotopies between componets of $K_0, K_1$ in $X_0, X_1$, respectively. In other words, $\Vert K_0 \Vert =\Vert K_t \Vert= \Vert K_1 \Vert$, so we are done in this case.\\

Now suppose $L_t$ is an irregular Legendrian isotopy (of $L_0$) with respect to $D(X)$ having only separating degenerations. As in the proof of Proposition \ref{prop:irregular_isotopy}, suppose there are numbers $0<t_0<t_1<\cdots <t_r<1$ so that except finitely many $L_{t_i},  (i=0,1,...,r)$, any other $L_t$ is an element of  $[L]_{min} \setminus [L]^o_{min}$. By assumption, during $L_t$ no births or deaths can arise, and only degenerations are separations. Recall that separations increase $\mathcal{P}_X$ and so $\widetilde{tb}$ by $1$. Therefore, one easily conclude that for any $i$ when passing from $t=t_i-\epsilon$ to $t=t_i+\epsilon$, there can not be any separations of knots in $K_{t_i-\epsilon}:=L_{t_i-\epsilon}\pitchfork D(X)$ because otherwise we would have 
$$\widetilde{tb}(K_0) \lneqq  \widetilde{tb}(K_{t_i+\epsilon}) \leq \widetilde{tb}(K_1).$$
\noindent So, $L_t$ must be a regular Legendrian isotopy indeed, and hence $\Vert K_0 \Vert = \Vert K_1 \Vert$ as discussed above. \\

\noindent \textit{Transivity}: Suppose $\Vert K_0 \Vert \preceq \Vert K_1 \Vert$ and $\Vert K_1 \Vert \preceq \Vert K_2 \Vert$ for $\Vert K_0 \Vert, \Vert K_1 \Vert, \Vert K_2 \Vert \in A$. Immediately, we observe $\widetilde{tb}(K_0)\leq \widetilde{tb}(K_1)$ and $\widetilde{tb}(K_1)\leq \widetilde{tb}(K_2)$, and so $$\widetilde{tb}(K_0)=\widetilde{tb}(K_2).$$
For each $i=0,1,2$, consider any representative $K_i \in \Vert K_i \Vert$ and the corresponding 
$L_i \in [L]_{min}\setminus [L]^o_{min}$. By assumption, there are Legendrian isotopies $L_t$ from $L_0$ to $L_1$ and 
$L'_t$ from $L_1$ to $L_2$ with the prescribed conditions in Definition \ref{def:partial_order} part (I). Then one easily concludes that $L'_t \circ L_t$ is a Legendrian isotopy from $L_0$ to $L_2$ with the desired properties. Thus, $\Vert K_0 \Vert \preceq \Vert K_2 \Vert$. 

\end{proof}

Returning back to the proof of the theorem, next we will show that every chain in $(A,\preceq)$ has an upper bound in $A$. To this end, suppose that we are giving a chain $$\Vert K_0 \Vert \preceq \Vert K_1 \Vert \preceq \Vert K_2 \Vert \preceq \cdots \preceq \Vert K_i \Vert \preceq \cdots .$$
Since regular Legendrian isotopies does not change the isotopy equivalence classes, it suffices to consider irregular Legendrian isotopy (with respect to $D(X)$) having only separating degenerations. We need to show that under such isotopies, separations must eventually stop after a finite step, and when it stops the corresponding $\widetilde{tb}$ must be finite.\\

Let $L_i$'s be Legendrian representatives in $[L]_{min} \setminus [L]^o_{min}$ such that, for each $i\geq 0$, we have $K_i=L_i \pitchfork D(X)$ and $L_{i+1}$ is the image of $L_i$ under an irregular Legendrian isotopy $L^t_i$ satifying the condition (I) of Definition \ref{def:partial_order}. Suppose the the link $K_i$ consists of $r_i$ knot components. (Recall by compactness there must be finite number of components for each $K_i$.) By Proposition \ref{prop:irregular_isotopy} and from the assumptions $K_i \in \Lambda$ and $L_i \in [L]_{min} \setminus [L]^o_{min}$, we know that each isotopy $L^t_i$ consists only of finitely many separations, and $\mathcal{P}_X$ (and so $\widetilde{tb}$) has a finite jump (increment) during each $L^t_i$. That is, we have 

\begin{center}
		$r_0 < r_1< r_2< \cdots <r_i< \cdots$
\end{center} 
with $0<r_{i+1}-r_i<\infty$, and
\begin{center}		
$\widetilde{tb}(K_0) < \widetilde{tb}(K_1)< \widetilde{tb}(K_2)< \cdots <\widetilde{tb}(K_i)< \cdots$
\end{center}
with $\widetilde{tb}(K_{i+1})-\widetilde{tb}(K_i)<\infty$.\\

Now observe that during the separations of any $L_i^t$, knot components in $K_i$ split into ``simpler'' knot components (which form the link $K_{i+1}$) which are still disjointly embedded simple closed curves in the resulting Legendrian suface $L_{i+1}$. From our choices, knot components in any $K_i$ can not bound disks in $L_i$ and can not be decomposable. Therefore, there must exist some $i_{max} \in \mathbb{N}$ such that we can not proceed further. That is, we have
\begin{center}
		$r_0 < r_1< r_2< \cdots <r_i< \cdots < r_{i_{max}}$
\end{center} 
where the sequence stops at $r_{i_{max}} <\infty$, and
\begin{center}		
$\widetilde{tb}(K_0) < \widetilde{tb}(K_1)< \widetilde{tb}(K_2)< \cdots <\widetilde{tb}(K_i)< \cdots < \widetilde{tb}(K_{i_{max}} )$
\end{center}
where the sequence stops at $\widetilde{tb}(K_{i_{max}} )<\infty$.\\

Therefore, every chain in $(A,\preceq)$ has an upper bound in $A$, and hence, by Zorn's Lemma, the partially ordered set $(A,\preceq)$ has at least one maximal element, say $\Vert K_{max} \Vert \in A$. Then by the definition of the partial order relation ``$\preceq$'', for a chosen representative $K_{max} \in \Vert K_{max} \Vert$, the number $\widetilde{tb}(K_{max})<\infty$ (exists) and is the maximum value among all possible values obtained from such links of transverse intersections. Then for a corresponding Legendrian representative, say $L_{max} \in [L]_{min} \setminus [L]^o_{min}$, one obtains $\mathcal{P}_X(L_{max})$ is finite and maximal among all, i.e., the relative invariant $M\mathcal{P}_X(L)=\mathcal{P}_X(L_{max})$ is well-defined.

\end{proof}

\section{Absolute Page Crossing Number and Proof of Theorem \ref{thm:absolute_page_crossing}} \label{sec:well-definedness_absolute}

Next we introduce page-free version of maximal page crossing number.

\begin{definition}
	Let $L\hookrightarrow (M^5, \xi)$ be a closed orientable Legendrian surface. Fix an admissable open book $(B, f)$ for $L$ \textbf{\textit{essentially intersecting}} $L$ which means that the double of every page of $(B, f)$ essentially intersects $L$. Fix any page $X$ of $(B, f)$. Then
	\begin{center}
		$M\mathcal{P}_{(B, f)}(L):=M\mathcal{P}_{X}(L)$
\end{center}
is called the \textbf{\textit{absolute maximal page crossing number of $L$ with respect to $(B, f)$}}.
\end{definition}

We start with the following fact which will be useful in proving well-definedness of $M\mathcal{P}_{(B, f)}(L)$:

\begin{lemma} \label{lem:perturbation_of_page}
	The relative invariant $M\mathcal{P}_{X}(L)$ does not change under (small) perturbations of the double $D(X)$ transverse to a Legendrian representative in $[L]$ realizing $M\mathcal{P}_{X}(L)$.
\end{lemma}

\begin{proof}
Suppose that $L'\in [L]$ realizes $M\mathcal{P}_{X}(L)$. In other words, $D(X)$ transversely intersects the Legendrian isotopic copy $L'$ of $L$ and we have $$M\mathcal{P}_{X}(L)=\mathcal{P}_{X}(L').$$
We want to show that $M\mathcal{P}_{X}(L)=M\mathcal{P}_{X'}(L)$ for any pair $X, X'$ of pages such that their doubles are isotopic to each other via a 1-parameter family of doubles of pages transverse to $L'$. Equivalently, need to show that $$\mathcal{P}_{X}(L')=\mathcal{P}_{X'}(L').$$ To this end, suppose  $X=X_{\theta}$, $X'=X_{\theta'}$ is such a pair of pages. Let $K=k_0 \cup k_1$ be any knot component of $L' \pitchfork D(X)$. Then as depicted in Figure \ref{fig:38} that a new knot component $K' \in L' \pitchfork D(X')$ is (Legendrian) isotopic to the older one $K$. So, the contribution of $K'$ to $\mathcal{P}_{X'}(L')$ is the same as the contribution of $K$ to the $\mathcal{P}_{X}(L')$. Thus, the claim follows.
\end{proof}
\begin{figure}[h!] 
	\includegraphics[scale=.78]{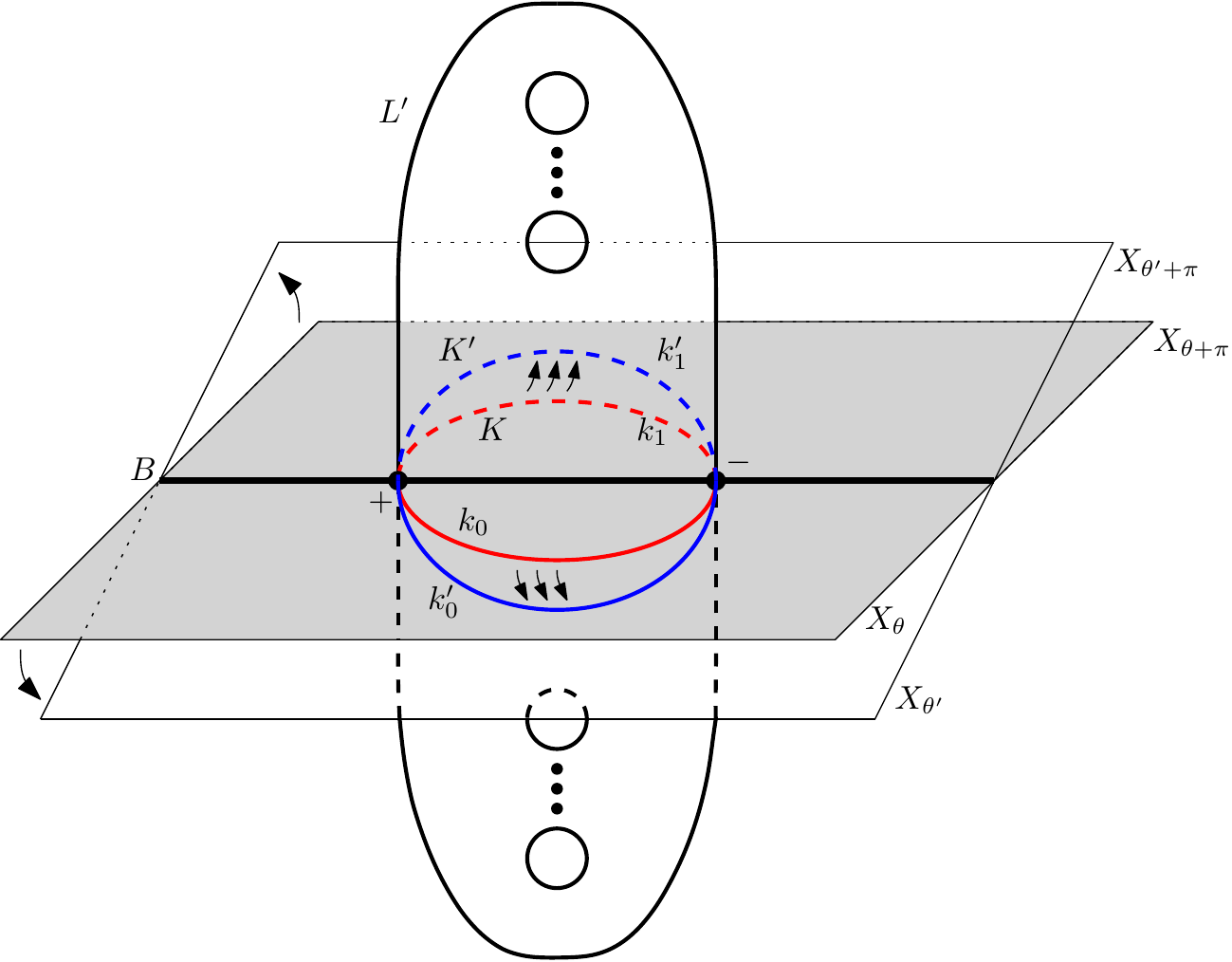}
	\caption{ Replacing $X_{0}=X_{\theta}$ (resp. $X_{1}=X_{\theta+\pi}$ ) with a nearby (Stein) page $X_{\theta'}$ (resp. $X_{\theta'+\pi}$ ) which are still intersecting $L'$ transversely. A new knot component $K'=k_{0}'\cup k_{1}'$ (isotopic to older one $K$) of the link of intersection of $L'$ with the new double $X_{\theta'}\cup_{\partial} X_{\theta'+\pi}$. }
	\label{fig:38}
\end{figure}	

Now with the help of the results from previous sections, one can easily prove the following:

\begin{proposition}
	 Let $L\hookrightarrow (M^5, \xi)$ be a closed orientable Legendrian surface. Fix an admissable open book $(B, f)$ for $L$. Then the number $M\mathcal{P}_{(B,f)}(L)$ is well-defined  and invariant under Legendrian isotopies of $L$.
\end{proposition}

\begin{proof}
Recall the definition of $M\mathcal{P}_{(B,f)}(L)$:

\begin{center}
		$M\mathcal{P}_{(B, f)}(L):=M\mathcal{P}_{X}(L)$
\end{center}
where $X$ is any page of $(B, f)$, or equivalently, 
\begin{center}
		$M\mathcal{P}_{(B, f)}(L):=M\mathcal{P}_{X_\theta}(L)$ for a fixed $\theta\in S^1$.
\end{center}

By assumption, $D(X_\theta)$ essentially intersects $L$ for all $\theta \in S^1$. For each  $\theta \in S^1$, by Lemma \ref{lem:perturbation_of_page}, the number $M\mathcal{P}_{X_\theta}(L)$ takes the same value on some small enough neighborhood $U_\theta$ of $\theta$ in $S^1$, and so the collection $\{U_\theta \; |\; \theta \in S^1 \}$ is an open cover for $S^1$. By compactness of $S^1$, there exists a finite  subcover, i.e., there exist $\theta_1, \theta_2, ..., \theta_r \in S^1$ such that $$S^1=U_{\theta_1} \cup U_{\theta_2} \cup \cdots \cup U_{\theta_r}.$$ After renaming (if necessary), one may assume that for any two consecutive arcs, we have $U_{\theta_i} \cap U_{\theta_{i+1}}\neq \emptyset$. Since $M\mathcal{P}_{X_{\theta_i}}(L)$ and $M\mathcal{P}_{X_{\theta_{i+1}}}(L)$ take constant values on their domains, they must agree on $U_{\theta_i} \cap U_{\theta_{i+1}}$, and hence $M\mathcal{P}_{X_\theta}(L)$ takes a constant value on $U_{\theta_i} \cup U_{\theta_{i+1}}$. Repeating the argument (by changing $i$ one by one), we conclude that $M\mathcal{P}_{X_\theta}(L)$ takes the same value on the whole $S^1$. Hence, $M\mathcal{P}_{(B, f)}(L) =M\mathcal{P}_{X_\theta}(L)$ is independent of $\theta$, and, in particular, is well-defined.\\

Finally, the fact that $M\mathcal{P}_{(B, f)}(L)$ is invariant under Legendrian isotopies just follows from its definition combined with Theorem \ref{thm:relative_page_crossing} (or Lemma \ref{lem:invariance_under_Legendrian_isotopy}).

\end{proof}


\section{An Example} \label{sec:Example}

Let $\mathbb{C}^3$ be the complex space with the complex coordinates $$(z_{1}, z_{2}, z_{3})=(r_1, \theta_1,r_2,\theta_2,r_3,\theta_3),$$ where $z_{j}=r_{j}e^{i\theta_{j}} (j=1, 2, 3)$ are the polar coordinates, and $\mathbb{S}^5$ be the unit $5$-sphere in $\mathbb C^3$, i.e.,
$$\mathbb{S}^5=\left\{ (z_{1},z_{2},z_{3}) \mid |z_{1}|^{2}+|z_{2}|^{2}+|z_{3}|^{2}=1 \right\}.$$

The restriction of the $1$-form (a primitive of the standard symplectic form $\omega_{st}$ on $\mathbb{C}^3$)
\begin{center} 
	$\alpha_{st} =r_{1}^{2}d\theta_{1}+r_{2}^{2}d\theta_{2}+r_{3}^{2}d\theta_{3}$.
\end{center}
on $\mathbb{S}^5$ is a contact form and defines the standard contact structure $\xi_{5}$ on $\mathbb{S}^5$. So we have a closed contact manifold $(\mathbb{S}^5, \xi_{5}=\textrm{Ker}(\alpha_{5}))$ where $\alpha_{5}:=\alpha_{st}|_{\mathbb{S}^5}$.\\
	
We will consider an open book supporting $\xi_5$ which is admissable for a Legendrian surface we pick later inside $(\mathbb{S}^5, \xi_{5})$. Consider the standard $3$-sphere
$$\mathbb{S}^3=\left\lbrace (z_{1}, z_{2}, z_{3}) \in \mathbb{S}^{5} \mid z_{1}=0 \right\rbrace \subset \mathbb{S}^{5}$$ 
with its standard contact structure $\xi_{st}$ as a contact submanifold of $\mathbb{S}^5$ as follows:
\begin{center}
		$(\mathbb{S}^3, \xi_{st}=\xi_{3}= \textrm{Ker}(\alpha_{3}) ) \hookrightarrow (\mathbb{S}^5, \xi_{5})$
\end{center}
where $\alpha_{3}=\alpha_{5}|_{\mathbb{S}^3}=r_{2}^{2}d\theta_{2}+r_{3}^{2}d\theta_{3}|_{\mathbb{S}^3}$ is the contact form on $\mathbb{S}^3$ (defining $\xi_{st}=\xi_{3}$). \\
	
Consider the most standard open book on $\mathbb{S}^{5}$ which can be explicitly described as follows:
	\begin{center}
		$\pi: \mathbb{S}^{5}\setminus \mathbb{S}^{3}\longrightarrow S^{1}$
		\\
		$(r_{1}, \theta_{1}, r_{2}, \theta_{2}, r_{3}, \theta_{3})\longmapsto \theta_{1}$.
	\end{center}
Note that the standard $\mathbb{S}^3$ is the binding, and a typical page $X_{\theta_{1}}=\pi^{-1}(\theta_{1})$ is an open $4$-ball (simply-connected and Weinstein). The closure of $X_{\theta_{1}}$ (still denoted by $X_{\theta_{1}}$ for simplicity) can be parametrized by
\begin{center}
	$X_{\theta_{1}}=\pi^{-1}(\theta_{1}):\left\{\vec\Gamma: D^{4}\longrightarrow S^{5}, \quad (\rho_{1},\phi_{1},\rho_{2},\phi_{2})\longmapsto(\sqrt{1-\rho_{1}^{2}-\rho_{2}^{2}}, \theta_{1},\rho_{1},\phi_{1},\rho_{2},\phi_{2} )\right\}$.
\end{center}
(Clearly, $X_{\theta_{1}}$ is diffeomorphic to $D^{4}$, and note $0\leq \rho_{1}^{2}+\rho_{2}^{2}\leq 1$.) One can easily check that the embedded open book $(\mathbb{S}^{3}, \pi)$ on $\mathbb{S}^{5}$ supports $\xi_{5}$ and the corresponding abstract open book is $(D^{4}, id_{D^{4}})$ (with a trivial monodromy).\\

Let's now pick a Legendrian surface $L$ inside $(\mathbb{S}^{5}, \xi_{5})$. For a fixed constant $k$, consider the \textbf{\textit{Clifford torus}} (a well-known and well-studied surface) defined by
\begin{center}
		$L=T_{k}=\left\{ (z_{1},z_{2},z_{3}) \in \mathbb{C}^{3} \mid |z_{1}|^{2}=|z_{2}|^{2}=|z_{3}|^{2}=\frac{1}{3},\quad \theta_{1}+\theta_{2}+\theta_{3}=k \right\} \subset  \mathbb{S}^{5}$
\end{center}
(note in polar coordinates we have $r_{1}^{2}=r_{2}^{2}=r_{3}^{2}=\frac{1}{3}$.) Clearly $T_k$ is a surface inside $\mathbb{S}^{5}$.  One needs to check that $T_k$ is Legendrian $(\mathbb{S}^{5}, \xi_{5})$. To this end, consider the following parametrization for $T_k$ where $\varphi_{1},\varphi_{2}$ are angular coordinates on an abstract torus $T^2$:
	\begin{center}
		$T_k: \quad \vec{\sigma}(\varphi_{1},\varphi_{2})=\big(\frac{1}{\sqrt{3}}, \varphi_{1},\frac{1}{\sqrt{3}}, \varphi_{2},\frac{1}{\sqrt{3}},k-\varphi_{1}-\varphi_{2}\big) \in \mathbb{S}^{5}$,\\\vspace{0.2cm}
		$\vec{\sigma}_{\varphi_{1}}=\langle0,1,0,0,0,-1\rangle=\dfrac{\partial}{\partial \theta_{1}}-\dfrac{\partial}{\partial \theta_{3}}$,\\\vspace{0.2cm}
		$\vec{\sigma}_{\varphi_{2}}=\langle0,0,0,1,0,-1\rangle=\dfrac{\partial}{\partial \theta_{2}}-\dfrac{\partial}{\partial \theta_{3}}$.\\
		\end{center}

Then, we easily see that 
\begin{center}
	$\alpha_{5}\mid_{T_{k}}=\frac{1}{3}d\theta_{1}+\frac{1}{3}d\theta_{2}+\frac{1}{3}d\theta_{3}$, \\\vspace{0.2cm}
	$\alpha_{5}\mid_{T_{k}}(\vec{\sigma}_{\varphi_{1}})=0=\alpha_{5}\mid_{T_{k}}(\vec{\sigma}_{\varphi_{2}})$.\vspace{0.2cm}
\end{center}

Therefore, $T_{k}$ is a Legendrian torus in $(\mathbb{S}^{5}, \xi_{5})$. Let's understand how $T_{k}$ intersects with the binding $\mathbb{S}^{3}$ and a typical page $X_{\theta_{1}}\approx D^4$: \\

\noindent For $T_{k}\cap \mathbb{S}^{3}$, we have
\begin{center}
		$\mathbb{S}^{3}=\left\lbrace (z_{1}, z_{2}, z_{3}) \in \mathbb{C}^{3} \mid z_{1}=0 \right\rbrace=\left\lbrace r_{2}^{2}+r_{3}^{2}=1, r_{1}=0\right\rbrace$.
\end{center}
But on $T_{k}$, $r_{1}=\frac{1}{\sqrt{3}}\neq 0$. Hence, $T_{k}\cap \mathbb{S}^{3}=\emptyset$. In particular, this shows that the binding of the open book $(\mathbb{S}^{3}, \pi)$ intersects $T_k$ transversely.\\

\noindent For $K_{\theta_{1}}:=T_{k}\cap X_{\theta_{1}}$,
\begin{center}
		$X_{\theta_{1}}:\left\{\vec{\Gamma}(\rho_{1},\phi_{1},\rho_{2},\phi_{2})=(\sqrt{1-\rho_{1}^{2}-\rho_{2}^{2}}, \theta_{1},\rho_{1},\phi_{1},\rho_{2},\phi_{2} )\right\}$,\\\vspace{0.2cm}
		$T_{k}:\vec{\sigma}(\varphi_{1},\varphi_{2})=\big(\frac{1}{\sqrt{3}}, \varphi_{1},\frac{1}{\sqrt{3}}, \varphi_{2},\frac{1}{\sqrt{3}},k-\varphi_{1}-\varphi_{2}\big)$.
\end{center} 

Equating the corresponding coordinates, one gets the equations defining the intersection $K_{\theta_{1}}$:

\begin{center}
	$\sqrt{1-\rho_{1}^{2}-\rho_{2}^{2}}=\frac{1}{\sqrt{3}}, \;\;\theta_{1}=\varphi_{1}, \;\;\rho_{1}=\frac{1}{\sqrt{3}}, \;\;\phi_{1}=\varphi_{2}, \;\;\rho_{2}=\frac{1}{\sqrt{3}}, \;\;\phi_{2}=k-\varphi_{1}-\varphi_{2}$.
\end{center}

If we let $\phi_{1}=\varphi_{2}=\theta$, then we obtain the parametrization of $K_{\theta_{1}}$ given by

\begin{center}
	$K_{\theta_{1}}: \quad \vec{r}:S^1 \to \mathbb{S}^5, \quad \vec{r}(\theta)=\big(\frac{1}{\sqrt{3}},\theta_{1},\frac{1}{\sqrt{3}}, \theta, \frac{1}{\sqrt{3}},k-\theta_{1}-\theta\big)$.
\end{center}
Note that the parameter $\theta$ appears in two distinct angular coordinates with opposite signs, and so $K_{\theta_{1}}$ is an embedded unknot in $\mathbb{S}^5$ sitting as a $(1,-1)$-torus knot on the Clifford torus $T_k$. Note that following the same steps, one can also  consider $K_{\theta_{1}+\pi}:=T_{k}\cap X_{\theta_{1}+\pi}$ which is also a $(1,-1)$-torus knot (so unknot) on the Clifford torus $T_k$ (a paralel copy of $K_{\theta_{1}}$). Hence, we conclude that the double $D(X)=X_{\theta_{1}} \cup_{\partial} X_{\theta_{1}+\pi}$ intersects $T_k$ transversely and essentially along the (un)link $D(X) \cap T_k=K_{\theta_{1}} \sqcup K_{\theta_{1}+\pi}$ for all $\theta_1 \in S^1$.\\

One may think (see the claim below) $K_{\theta_{1}}$ $K_{\theta_{1}+\pi}$ as Legendrian unknots in $(\mathbb{S}^{3}, \xi_3)$ . For $K_{\theta_{1}}$:
\begin{center}
	$\vec{r'}(\theta)=\langle0,0,0,1,0,-1\rangle=\dfrac{\partial}{\partial \theta_{2}}-\dfrac{\partial}{\partial \theta_{3}}$,\\\vspace{0.2cm}
	$\alpha_{3}\mid _{K_{\theta_{1}}}=\frac{1}{3}d\theta_{2}+\frac{1}{3}d\theta_{3}$
\end{center}
and so, $\alpha_{3}\mid _{K_{\theta_{1}}}(\vec{r'}(\theta))=0$. Verification for $K_{\theta_{1}+\pi}$ is similar. Note that $K_{\theta_{1}}$ and $K_{\theta_{1}+\pi}$ can be also considered as Legendrian unknots in the Stein diagrams of $X_{\theta_{1}}$ and $X_{\theta_{1}+\pi}$, respectively.\\
		
Indeed, we have
	\begin{center}
		$\alpha_{5}\mid _{X_{\theta_{1}}}=\vec{\Gamma}^{*}(\alpha_{5})=\underbrace{(1-\rho_{1}^{2}-\rho_{2}^{2})\wedge d\theta_{1}}_{=0 \;(\theta_1 \textbf{ fixed})} +\rho_{1}^{2}d\phi_{1}+\rho_{2}^{2}d\phi_{2}=\rho_{1}^{2}d\phi_{1}+\rho_{2}^{2}d\phi_{2}$, and so\\\vspace{0.2cm}
		$d\alpha_{5}\mid_{X_{\theta_{1}}}=d(\alpha_{5}\mid _{X_{\theta_{1}}})=d(\rho_{1}^{2}d\phi_{1}+\rho_{2}^{2}d\phi_{2})=2\rho_{1}d\rho_{1}\wedge d\phi_{1}+2\rho_{2}d\rho_{2}\wedge d\phi_{2}$
	\end{center}
from which we compute
\begin{center}
	$d(\alpha_{5}\mid _{X_{\theta_{1}}})\mid _{K_{\theta_{1}}}=d(\alpha_{5}\mid _{T_{k}\cap X_{\theta_{1}}})=\frac{2}{\sqrt{3}}d\rho_{1}\wedge d\phi_{1}+\frac{2}{\sqrt{3}}d\rho_{2}\wedge d\phi_{2}$, and also \\\vspace{0.2cm}
	$(\alpha _{5}\mid _{X_{\theta_{1}}})|_{K_{\theta_{1}}}=\frac{1}{3}d \phi_{1}+\frac{1}{3}d \phi_{2}(=\frac{1}{3}d \theta_{2}+\frac{1}{3}d \theta_{3})$\\\vspace{0.2cm}
	$\Rightarrow (\alpha _{5}\mid _{X_{\theta_{1}}})|_{K_{\theta_{1}}}(\vec{r'}(\theta))=\frac{1}{3}-\frac{1}{3}=0$.
\end{center}	

These verify that $K_{\theta_{1}}, K_{\theta_{1}+\pi}$ are isotropic unknots in every (simply-connected) Weinstein (so Stein) page $(X_{\theta_{1}}, d\alpha_{5}\mid _{X_{\theta_{1}}})$, and every page of the open book $(\mathbb{S}^{3}, \pi)$ essentially intersects $T_k$. As a result, we conclude that $(\mathbb{S}^{3}, \pi)$ is an essentially intersecting admissable open book for the Clifford torus $T_k$. See Figure \ref{fig:Cliff_front} for a schematic picture for the front projection of $T_k$.

\begin{figure}[h]
	\centering
	\includegraphics[scale=.8]{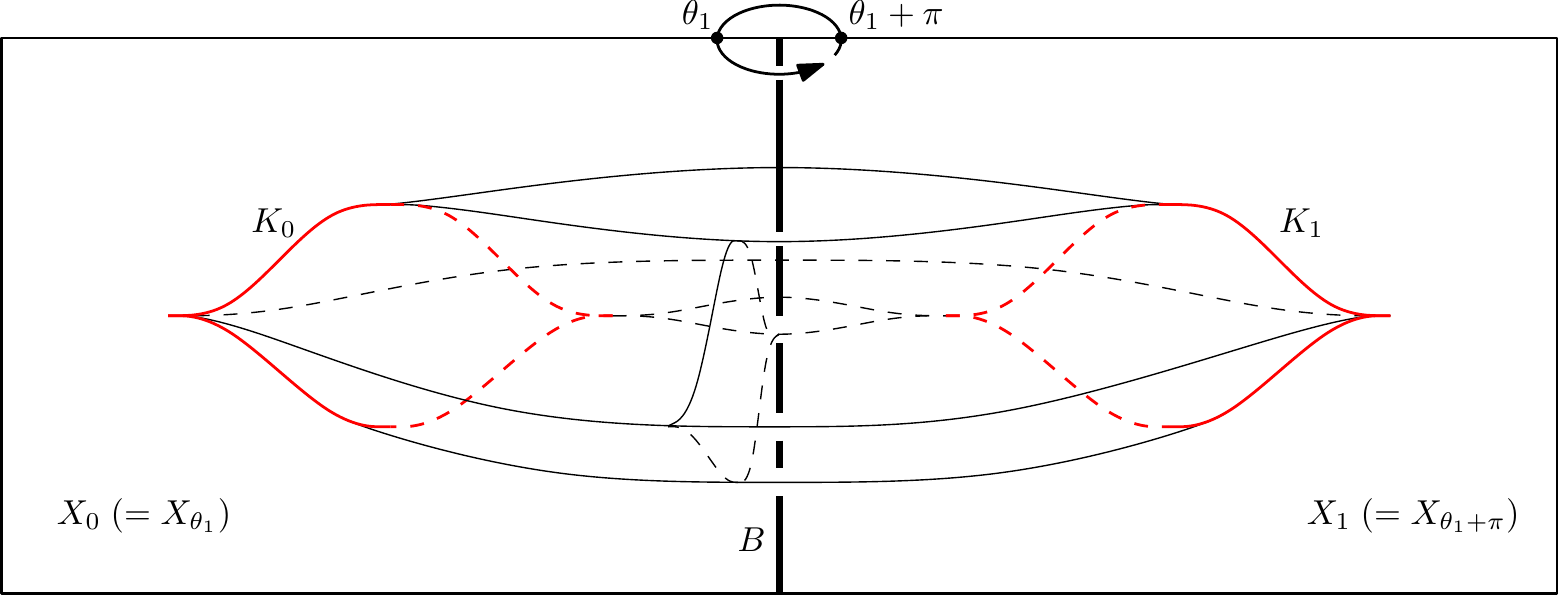}
	\caption{A schematic picture of the front projection of the Legendrian Clifford torus $T_k \subset (\mathbb{S}^{5}, \xi_{5})$ onto $\mathbb{R}^3$ with coordinates $z,y_1,y_2$. The components $K_{0}, K_{1}$  (in red) of the (un)link of intersection of $T_k$ with the double $X_0 \cup_{\partial} X_1\approx \mathbb{S}^4$ of a page $X\approx D^4$ of the trivial open book on $\mathbb{S}^5$. (Note: $B\approx \mathbb{S}^3$ and $T_k \pitchfork B=\emptyset$.)}
		\label{fig:Cliff_front}
\end{figure}

Next, we set  $K_0=K_{\theta_{1}}, K_1=K_{\theta_{1}+\pi}$ and also $X_{0}=X_{\theta_{1}}$ and $X_{1}=X_{\theta_{1}+\pi}$, so that the (un)link of intersection $T_k$ with the double $D(X)=X_{0}\cup_{\partial} X_{1}$ is given by
\begin{center}
	$T_k \pitchfork D(X)=K_0 \sqcup K_1$
\end{center}
With respect to the notation introduced in Section \ref{sec:Relative_Page_Crossing}, we have $K_{0}=k_{0}^{0} \subset X_0$ (no $k_{0}^{1}$) and $K_{1}=k_{1}^{1} \subset X_1$ (no $k_{1}^{0}$). Also recall $T_k\cap \mathbb{S}^{3}=\emptyset$. Hence, $T_k\pitchfork D(X)$ is an unlink with two components $k_{0}^{0}, k_{1}^{1}$ which can be realized in Stein diagrams of $(X_{0}, d\alpha_{5}|_{X_0})$ and $(X_{1}, d\alpha_{5}|_{X_{1}})$ as in Figure \ref{fig:Cliff_Stein_diag}. (Of course, one can make cancelations to obtain simpler diagrams... )

\begin{figure}[h!]
	\centering	\includegraphics[scale=1]{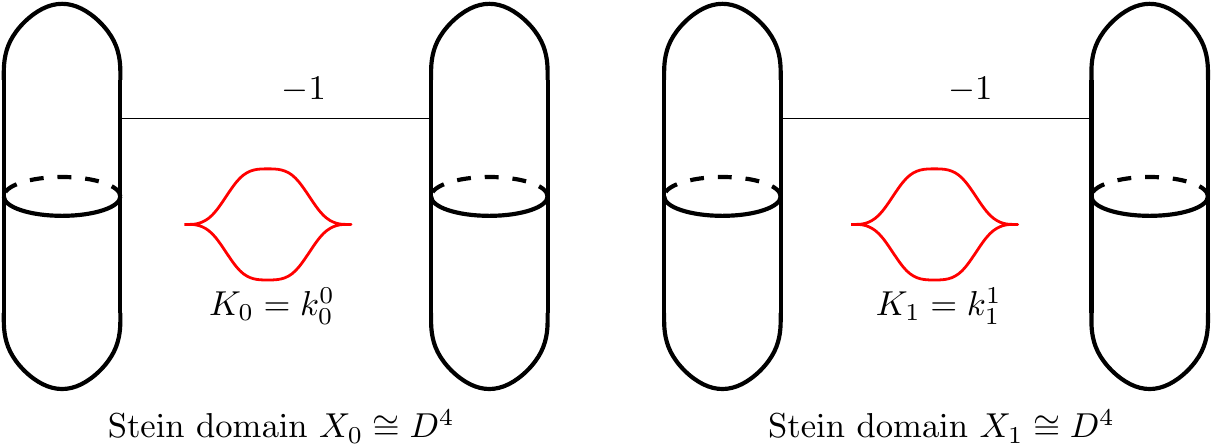}
	\vspace{0.5cm}
	\caption{Realizing the (un)link of transverse intersection of the Legendrian Clifford torus $T_k$ with the double $X_0 \cup_{\partial} X_1\approx \mathbb{S}^4$ of a page $X\approx D^4$ in the Stein diagrams of $(X_{0}, d\alpha_{5}|_{X_0})$ and $(X_{1}, d\alpha_{5}|_{X_{1}})$. }
	\label{fig:Cliff_Stein_diag}
\end{figure}

Here one needs to verify that Figure \ref{fig:Cliff_Stein_diag} reflects the correct pictures of $K_{0}=k_{0}^{0}$ and $K_{1}=k_{1}^{1} $:

\begin{claim}
$tb(K_0)=-1$ ( and so is $tb(k_{0}^{0})$ ) \, and \;\;$tb(K_1)=-1$ ( and so is $tb(k_{1}^{1})$ ). 
\end{claim}

\begin{proof}
It suffices to show that $tb(K_0)=-1$ (since the proof of $tb(K_1)=-1$ follows exactly the same steps with different labels). As observed above,  $K_0$ is an isotropic unknot in the Stein page $(X_0, d\alpha_{5}\mid _{X_{0}})$. Consider the Liouville vector field $\chi_{st}=\frac{1}{2}(r_1\partial_{r_1}+r_2\partial_{r_2}+r_3\partial_{r_3})$ of $\omega_{st}=d\alpha_{st}$ which restricts on $X_0$ to $$\chi_0=(\dfrac{r_1^2-1}{2r_1})\partial_{r_1}+\frac{r_2}{2}\partial_{r_2}+\frac{r_3}{2}\partial_{r_3}=\vec{\Gamma}_*(\frac{\rho_1}{2}\partial_{\rho_1}+\frac{\rho_2}{2}\partial_{\rho_2})$$ which is a Liouville vector field for $d\alpha_{5}\mid _{X_{0}}$. (Here $\partial_{r_i}=\partial / \partial_{r_i}, \partial_{\rho_i}=\partial / \partial_{\rho_i}$, ... etc.) We compute the flow map $H_t:X_0 \to X_0$ of $\chi_0$ as:
$$H_t(r_1,\theta_1,r_2.\theta_2,r_3,\theta_3)=\left(\sqrt{1-(1-r_1^2)e^t}, \theta_1, r_2e^{t/2}, \theta_2, r_3e^{t/2}, \theta_3\right). $$ Observe that $H_{\ln(3/2)}$ maps any point $(\frac{1}{\sqrt{3}},\theta_1,\frac{1}{\sqrt{3}},\theta_2,\frac{1}{\sqrt{3}},\theta_3)$ in the interior of $X_0$ to a point $(0,\theta_1,\frac{1}{\sqrt{2}},\theta_2,\frac{1}{\sqrt{2}},\theta_3)$ in the contact boundary $(\mathbb{S}^{3}, \xi_3)$ of $(X_0, d\alpha_{5}\mid _{X_{0}})$. In particular, the image $H_{\ln(3/2)}(K_0)$ is a Legendrian unknot in  $(\mathbb{S}^{3}, \xi_3)$ which we will still denote by $K_0$. As a result, we may think of $K_0$ (drawn in the Stein diagram of $X_0$ ) as a Legendrian unknot $K_0 \subset (\mathbb{S}^{3}, \xi_3)$ with the Legendrian embedding

$$K_{0}: \quad \vec{\beta}:S^1 \to \mathbb{S}^3, \quad \vec{\beta}(\theta)=\left(0,\theta_{1},\frac{1}{\sqrt{2}}, \theta, \frac{1}{\sqrt{2}},k-\theta_{1}-\theta\right) \quad(\theta_1 \textbf{ fixed}).$$

We note that such an understanding of a knot is a common method in handle decomposition theory of 4-manifolds, and for Legendrian knots in Stein surfaces it is carefully studied in \cite{G8}. 

In the rest of the proof, for the computational purposes, we will keep using polar coordinates $(r_1,\theta_1)$ but switch to cartesian coordinates $(x_2,y_2,x_3,y_3)$ in the last four coordinates. In the new coordinates, by restricting the Liouville form 
$\alpha_{st} =r_{1}^{2}d\theta_{1}+x_2dy_2-y_2dx_2+x_3dy_3-y_3dx_3$ on $\mathbb{C}^3$ to the binding $\mathbb{S}^{3}$ ($r_1=0$), we obtain the contact form $\alpha_3$ defining $\xi_3$ and its Reeb vector field $R_{\alpha_3}$ given as

$$\alpha_3 =x_2dy_2-y_2dx_2+x_3dy_3-y_3dx_3, \quad R_{\alpha_3}=\langle 0,0 -y_2,x_2,-y_3,x_3 \rangle.$$

In the new coordinates the Legendrian unknot $K_0 \subset (\mathbb{S}^{3}, \xi_3)$ has the parametrization

$$K_{0}: \quad \vec{\beta}:[0,2\pi] \to \mathbb{S}^3, \quad(\theta_1 \textbf{ fixed})$$

$$\vec{\beta}(\theta)=
\left(0,\theta_{1},\frac{1}{\sqrt{2}}\cos(\theta), \frac{1}{\sqrt{2}}\sin(\theta), \frac{1}{\sqrt{2}} \cos(k-\theta_{1}-\theta),\frac{1}{\sqrt{2}}\sin(k-\theta_{1}-\theta)\right),$$

$$\vec{\beta}'(\theta)=
\frac{1}{\sqrt{2}} \big\langle 0,0,-\sin(\theta), \cos(\theta), \sin(k-\theta_{1}-\theta),-\cos(k-\theta_{1}-\theta) \big\rangle).$$

Also the restriction of the Reeb vector field on the $K_0$, in the coordinates $(r_1,\theta_1,x_2,y_2,x_3,y_3)$, is given by 

$$R_{\alpha_3} \mid_{K_0}=
\frac{1}{\sqrt{2}} \big\langle 0,0,-\sin(\theta), \cos(\theta), -\sin(k-\theta_{1}-\theta),\cos(k-\theta_{1}-\theta)\big\rangle.$$

Now observe that along $K_0$, the Reeb vector field $R_{\alpha_3} \mid_{K_0}$ makes exactly one full left twists. (It is enough to keep track of the last two components.) That is, if $K_0'$ denotes the parallel (contact) push-off of $K_0$ along $R_{\alpha_3} \mid_{K_0}$, then we have $lk(K_0,K_0')=-1$. Equivalently, $$tb(K_0)=-1$$ as claimed.

\end{proof}

Finally, after the above verification, we can calculate our invariants by using the Stein diagrams as follows: Using the notations introduced, we have
\begin{center}
	$\widetilde{tb}(K_{0})=tb(k_{0}^{0})=-1$  \quad  and  \quad   $\widetilde{tb}(K_{1})=tb(k_{1}^{1})=-1$.
\end{center}
So, the page crossing number $\mathcal{P}_{X}(L)$ for any page $X$ of the open book $(\mathbb{S}^{3}, \pi)$ is computed as
\begin{center}
	$\mathcal{P}_{X}(L)=\widetilde{tb}(K_{0})+\widetilde{tb}(K_{1})=-2$.
\end{center}
This is because with respect to any page, minimal link of intersection set has always two components as depicted in Figure \ref{fig:Cliff_front} and Figure \ref{fig:Cliff_Stein_diag}.) Also note that among all Legendrian representatives in $[T_k]$, the above embedding $\vec{\sigma}$ of $L=T_k$ gives the maximum possible value for $\mathcal{P}_{X}$ due to the Bennequin inequality (every Legendrian unknot bounds a disk in a tight three-sphere). 
As a result, the absolute and relative maximal page crossing numbers are computed as
\begin{center}
	$M\mathcal{P}_{(S^{3}, \pi)}(T_k)=M\mathcal{P}_{X}(T_k)=-2$.
\end{center} 

\begin{remark}
Thurston-Bennequin number of any Legendrian torus $L$ (regardless of how it is embedded in $(\mathbb{S}^{5}, \xi_{5})$) is computed as $tb(L)=0$ since it coincides with a topological invariant (see \cite{EES5}). So it is not possible to distinguish such Legendrian tori using Thurston-Bennequin invariant. On the other hand, since the new invariants defined here keep track Legendrian embeddings, they distinguish not only smooth embedding types of Legendrian surfaces but also their Legendrian isotopy types.
\end{remark}


\addcontentsline{toc}{chapter}{\textsc{References}}

\end{document}